\documentclass[reqno]{amsart}

\usepackage{amsmath,amsfonts, amssymb}
\usepackage{epsfig}
\usepackage{tikz}
\usepackage{wrapfig}
\usepackage{chngcntr}
\usepackage{apptools}

\def\bra{\langle}
\def\ket{\rangle}

\def\N{{\mathbb N}}

\def\R{{\mathbb R}}
\def\Z{{\mathbb Z}}

\def\cE{{\mathcal E}}
\def\cI{{\mathcal I}}

 \def\cK{{\mathcal K}}
 \def\bB{{\bf B}}
  
\def\eqnn{\begin{eqnarray*}}
\def\eeqnn{\end{eqnarray*}}
\def\eqn{\begin{eqnarray}}
\def\eeqn{\end{eqnarray}}

\theoremstyle{plain}
\newtheorem{theorem}{Theorem}[section]
\newtheorem{definition}[theorem]{Definition}
\newtheorem{proposition}[theorem]{Proposition}

\newtheorem{lemma}[theorem]{Lemma}

\theoremstyle{remark}
\newtheorem{remark}{Remark}
\newenvironment{prf}{{\it Proof:}}{$\Box$}

\numberwithin{equation}{section}

\begin{document}

\parskip=8pt
\parindent=0pt
\abovedisplayskip=12pt plus 3pt minus 9pt
\abovedisplayshortskip=0pt plus 3pt
\belowdisplayskip=12pt plus 3pt minus 9pt
\belowdisplayshortskip=7pt plus 3pt minus 4pt


\title[Mittag-Leffler moments to solutions of BTE]
{\bf{On Mittag-Leffler moments for the  Boltzmann equation for hard potentials without cutoff}}

\author[M. Taskovi\'{c}]{Maja Taskovi\'{c}}
\address{M. Taskovi\'{c},
Department of Mathematics, University of Pennsylvania.}
\email{taskovic@math.upenn.edu}

\author[R. J. Alonso]{Ricardo J. Alonso}
\address{R. Alonso,  
Department of Mathematics, PUC-Rio.}
\email{ralonso@mat.puc-rio.br}

\author[I. M. Gamba] {Irene M. Gamba}
\address{I. M. Gamba,
Department of Mathematics, The University of Texas at Austin.}
\email{gamba@math.utexas.edu}

\author[N. Pavlovi\'{c}]{Nata\v{s}a Pavlovi\'{c}}
\address{N. Pavlovi\'{c},  
Department of Mathematics, The University of Texas at Austin.}
\email{natasa@math.utexas.edu}

\begin{abstract}
We establish the $L^1$ weighted propagation properties for solutions of the Boltzmann equation with hard potentials and non-integrable angular components in the collision kernel. Our method identifies null forms by angular averaging and deploys moment estimates of solutions to the Boltzmann equation whose summability is achieved by introducing the new concept of Mittag-Leffler moments - extensions of $L^1$ exponentially weighted norms.
Such $L^1$ weighted norms of solutions to the Boltzmann equation are, both, generated and propagated in time and the characterization of their corresponding Mittag-Leffler weights depends on the angular singularity and potential rates in the collision kernel. These estimates are a fundamental step in order to obtain $L^\infty$ exponentially weighted estimates for solutions of the Boltzmann equation being developed in a follow up work.

\end{abstract}

\maketitle
\section{Introduction} \label{sec-intro}

We study generation and propagation in time of  $L^1$  exponentially weighted norms, referred to as exponential moments, associated to probability density functions that solve the Boltzmann equation \cite{bo872, bo64} modeling the evolution of monoatomic rarefied gases. Binary interactions of gas particles are described by transition rates from before and after such interactions, usually referred to as collision kernels. Such kernels are modeled as a product of potential functions of local relative speed and functions of the scattering angle between the pre and post relative velocities. This angular function may or may not be integrable.  When integrable, the collision kernel is said to satisfy an angular cutoff condition. The particular case when the angular part of the kernel is bounded, is known as the Grad's cutoff condition \cite{gr63}. Otherwise, its non-integrability, referred to as an angular non-cutoff, satisfies specific conditions (for details see Section 2).

The concept of exponential moments  is associated to the notion of large energy decay rates for tails. A time dependent probability distribution function $f(t,v)$ is said to have $L^1$ exponential moment (tail behavior) of order $s$ and rate $r(t)$ if, for any fixed $t>0$,
\begin{align}\label{tt}
 \int_{\R^d} f(t,v) e^{r(t) \; \langle v\rangle^s \; dv} 
 \ \ \ \ \mbox{is  positive and finite.}
\end{align}
This concept was introduced  by Bobylev in \cite{bo84, bo97} and Gamba, Panferov and Villani in \cite{gapavi09}, where they show uniform in time propagation of $L^1$ Maxwellian tails  (i.e. Gaussian in $v$-space, that is $s=2$) for several type of collision kernels ranging from Maxwell-type  to hard sphere interactions with angular cutoff conditions, and by Bobylev, Gamba, Panferov in \cite{bogapa04} for different values of $s\in(0,2]$ in the study of inelastic interaction with internal heating sources.
These ground breaking works conceived the idea of controlling exponential moments by proving the summability of power series expansions on a parameter $r(t)$. Such formulation was motivated by formally commuting integration in $v$-space and the infinite sum derived from the power series of the exponential function in \eqref{tt}, upon which one obtains
 \begin{align} \label{tt1}
 \int_{\R^d} f(t,v) \sum_{q=0}^\infty \frac{r^q(t) \langle v\rangle^{sq}}{\Gamma(q+1)} 
	\; = \; \sum_{q=0}^\infty \frac{r^q(t) \; m_{sq}(t)}{\Gamma(q+1)}.
\end{align}
The terms $m_{sq}(t)$, called polynomial moments,  are $\langle v \rangle^{sq}$ -weighted $L^1$    norms of the distribution function $f(v,t)$ that solves the Boltzmann equation.  Representation \eqref{tt1}  replaces the quest of  $L^1$ exponential integrability with a given order and rate,  with study of summability of infinite sums (time series forms).

A fundamental technique for accomplishing this task (see \cite{bo97, bogapa04, gapavi09, alga08,allo13, mo06}) consists of controlling the weak form of the collision operator by the means of angular averaging. These estimates are used to derive a sequence of ordinary differential inequalities for the polynomial moments of the collisional form. These differential inequalities are an algebraic sum of a negative term of moments of highest order and a positive term of bilinear sums of moments of lower orders.

Recently  Alonso, Ca\~nizo, Gamba and Mohout \cite{alcagamo13} introduced a new technique (based on analyzing partial sums corresponding to the infinite sum appearing in \eqref{tt1}), to prove the generation of  exponential moments with  orders up to the potential rate and  the  propagation of exponential moments with orders up to $s=2$, under an angular integrability condition. 
It is interesting to note that these results do not rely on the rate of Povzner estimates for angular averaging, and so the resulting order $r(t)$ may not be optimal.

All results mentioned above were developed for the case of an integrable angular collision kernel. This brings us to the setting of this manuscript, the non-cutoff regime. 
 This manuscript focuses on the study of both generation and propagation in time of  exponential moments for solutions to the initial value problem for the $d$-dimensional Boltzmann equation for elastic collisions, in the space homogeneous case, for hard potentials without the angular cutoff assumption.
In this direction, Lu and Mouhot  \cite{lumo12} showed generation of exponential moments of order up to the potential rate in the collision kernel. In this work, we considerably extend their result by showing that rates and orders of exponential moments depend on the initial data, as well as  potential and angular singularity rates in collision kernels.

In order to treat the non-cutoff regime, we develop angular averaged estimates that account for the cancellation of non-integrable angular singularities by means of null forms averaging.
The other important component is summability of moments, which is achieved  by introducing Mittag-Leffler moments.

Indeed, the most significant point of this paper is the introduction of Mittag-Leffler moments, as $L^1$ Mittag-Leffler weighted norms. They enabled us to extend the range of orders of exponential moments that can be propagated uniformly in time for the non-cutoff case. To obtain our result, we encounter the need to study \eqref{tt1} where $\Gamma(q+1)$ is replaced with $\Gamma(aq+1)$, for a noninteger $a>1$ (which is reminiscent of some of the tools used in \cite{bogapa04}, although no summing of such renormalized moments was performed there)
\begin{align} \label{tt1'}
 \int_{\R^d} f(t,v) \sum_{q=0}^\infty \frac{r^q(t) \langle v\rangle^{sq}}{\Gamma(aq+1)} 
	\; = \; \sum_{q=0}^\infty \frac{r^q(t) \; m_{sq}(t)}{\Gamma(aq+1)}.
\end{align}
 We observed that the sum appearing on the left-hand side of \eqref{tt1'} is exactly the well-known Mittag-Leffler function $\cE_a (r(t) \langle v \rangle^s)$, where  $\cE_a$ is defined as
\eqn
	\cE_a (x) := \sum_{q=0}^\infty \frac{x^q}{\Gamma(aq+1)}.
\eeqn
In analogy to \eqref{tt}, this led us to introduce a concept of Mittag-Leffler moments
\begin{equation}
			\displaystyle\int_{\R^d}  f(t,v) \;\; \cE_{a} (\alpha^{a} \, \bra v \ket^2 ) \; dv
		\;\; = \;\; \displaystyle \sum_{q=0}^\infty \frac{m_{2q}(t) \, \alpha^{aq}}{\Gamma(aq+1)},	
\end{equation}
which are a natural generalization of exponential moments.

Another important aspect of our main result is that the highest order of exponential moment which can be propagated in time, depends continuously on the singularity rate of the angular cross-section. The less singular the angular kernel is, the higher order exponential moment can be propagated. See details in Remark~\ref{mainrmk}.

Let us mention one application of $L^1$ weighted estimates. In \cite{gapavi09}, Gamba, Panferov and Villani, gave a proof to close the open problem of propagation of $L^\infty$-Maxwellian weighted bounds, uniformly in time, to solutions of the Boltzmann equation with hard potential with a cutoff in the angular kernel.
 Their result follows from an application of  a maximum principle of parabolic type, due to the dissipative nature of the collisional integral, and  estimates on the Carleman representation of the gain (positive) part of the collision operator 
 that depend on the  $L^1$-Maxwellian weighted bounds  uniformly  propagated in time.   
We mention here that the extension of such result on propagation of $L^\infty(\R^d)$-exponential weights is currently being worked out for the non-cutoff and hard potential case in a forthcoming manuscript \cite{gapata15} using the $L^1$ weighted estimates obtained in this manuscript.

\subsection*{Organization of the paper} Section \ref{sec-prelim}  presents the  Boltzmann equation without the angular cutoff condition, exponential and Mittag-Leffer moments and the statements of the two main results of the manuscript - the angular averaged Povzner inequalities with angular singularity cancellation in Lemma~\ref{averaged-pov},   and  the generation and propagation of Mittag-Leffler moments in Theorem~\ref{thm}. Section 3 contains the proof of the angular averaged Povzner inequalities for non-integrable angular singularity, i.e. Lemma~\ref{averaged-pov}. This lemma is the main tool for the formation of ordinary differential inequalities for polynomial moments of all orders, which are covered in Section 4.   Section 5 provides  details of the proof of the propagation of Mittag-Leffler moments, while in Section 6 we give a new proof of the generation of exponential moments of order up to the rate of potentials. The final section, Appendix, gathers known and technical yet fundamental results used throughout this manuscript.

\section{Preliminaries and Main Results}\label{sec-prelim}

\subsection{The Boltzmann equation}
\noindent We consider the Cauchy problem for the spatially homogeneous (i.e. $x$-space independent) Boltzmann equation
 \eqn \label{eq-be}
	\left\{
		\begin{array}{l}
			\partial_t f(t,v) \, = \, Q(f,f)(t,v),  \quad t\in\R^+, v\in\R^d, \quad d\geq 2 \vspace{5pt}\\ 
			f(0,v) \, = \, f_0(v).
		\end{array}
	\right. 
\eeqn
\noindent The function  $f(t,v)$ models the particle density at time $t$ and velocity $v$ of a rarefied gas in which particle collisions are elastic and predominantly binary. The collisional operator $Q(f,f)$ is a quadratic integral operator  defined via
\eqn \label{eq-q}
	Q(f,f)(t, v) \, = \, \int_{\R^d} \int_{S^{d-1}} \, \big( f'  f'_* - f f_* \big) \, B(|u|, \hat{u}\cdot\sigma) \, d\sigma \, dv_*,
\eeqn
\noindent where we use the abbreviated notation $f_*= f(t,v_*)$,
 $f'=f(t,v')$, and $f'_*=f(t,v'_*)$.  Vectors $v', v'_*$ denote pre-collisional velocities and $v, v_*$ are their corresponding post-collisional velocities. Relative velocity is denoted by $u=v-v_*$, and its normalization by $\hat{u}=u/|u|$. Being an elastic interaction of reversible character that conserves momentum $v+v_* =v' +v'_*$ and energy $|v|^2 + |v_*|^2 = |v'|^2 + |v'_*|^2$,  pre and postcollisional velocities are related by formulas represented in center of mass $V={(v+v_*)}/{2}$ and relative velocity $u=v-v_*$ coordinates  as follows
\eqn\label{velo-pre}
	v \, = \,  V'\, + \,  \frac{|u'|}{2}\; \sigma,
		\quad  \quad
		v_* \, = \, V' \, - \,  \frac{|u'|}{2} \; \sigma,
		\quad \quad 
\sigma \in S^{d-1}.
\eeqn

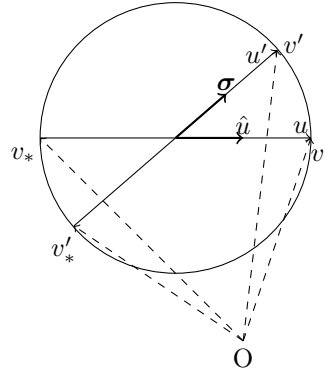
\begin{wrapfigure}{r}{0.35 \textwidth}
\setlength{\unitlength}{1cm}
\begin{tikzpicture}[scale=0.45]
\centering
\begin{scope}
   	\draw (0,0) circle (4);
	\draw [->] (-4,0)-- (4,0);
	\node at (-4.5, -0.5) {$v_*$};
	
	\node at (3.7, 0.3)  {$u$};
	\node at (2.4, 2.5)  {$u'$};
	
	\draw [->] [thick] (0,0)--(2,0);
	\node at (2, .4)  {$\hat{u}$};
	\node at (4.2, -0.5) {$v$};
	\node at (3.5, 2.8) {$ v'$};

	\draw [->] (-3,-2.6) -- (3,2.6);
	\node at (-3.3, -3.3)  {$v'_*$};

	\draw [->] [thick] (0,0)--(1.5, 1.3);
	\node at (1.5, 1.6)  {$\boldsymbol{\sigma}$};

	\draw[->] [dashed] (2,-6)--(4,0);
	\draw[->] [dashed] (2,-6)--(-4,0);
	\draw[->] [dashed] (2,-6)--(-3,-2.6);
	\draw[->] [dashed] (2,-6)--(3,2.6);
	\node at (2, -6.5) {O};
\end{scope}
\centering
\end{tikzpicture}
\caption [] {\hskip -0.1in \begin{minipage}[c]{1.\linewidth} Pre-post collisional \\ velocities  \end{minipage}}
\vspace{-0.4in}
\end{wrapfigure}

The unit vector $\sigma \in S^{d-1} $, referred to as the scattering direction, has the direction of the pre-collisional relative velocity $u'=v'-v'_*$. We  bring to the reader's attention that the pre to post  collisional  exchange of coordinates satisfy 
\begin{flalign*}
& v' - v \, =  \,  \frac12(|u|\; \sigma -u),\\ 
& v'_*  - v_*\, = \, -\frac12 (|u|\; \sigma -u). &
\end{flalign*}
This representation embodies the relation  of the exchange of velocity directions as just functions of the relative velocity $u$ and the scattering direction $\sigma$.

\bigskip
The collisional kernel $B(|u|, \hat{u}\cdot\sigma)$ is assumed to take the form
\begin{align}\label{eq-kernel}
	B(|u|, \hat{u}\cdot\sigma) \, = \, |u|^\gamma \, b(\cos{\theta}),
\end{align}
where $\theta \in [0,\pi]$ is the angle between the pre and post collisional relative velocities, and thus it satisfies 
$\cos\theta = \hat u \cdot \sigma$.  In this manuscript we work in the variable hard potentials case, that is
\begin{align} \label{eq-gamma}
0<\gamma\leq 1.
\end{align}

We assume that the angular kernel is given by a positive measure $b(\hat{u} \cdot \sigma) $ over the sphere $S^{d-1}$. In many models, this function is non-integrable over the sphere, while its weighted integral is finite. In this manuscript we assume that for some $\beta \in (0,2]$ the following weighted integral is finite (with  $V_{d-2}= \frac {\pi^{(d-2)/2}}{\Gamma((d-1)/2)}$ being the volume of the $d-2$ dimensional unit sphere)
\begin{align}\label{eq-nc} 
A_\beta  &: =  \int_{S^{d-1}} b(\hat{u} \cdot \sigma) \sin^{\beta}{\theta}  \; d\sigma \nonumber \\
& = V_{d-2}  \int_0^\pi b(\cos{\theta})   \; \sin^{ \beta}{\theta} \; \sin^{d-2}{\theta}\; d\theta < \infty.
\end{align}
When $\beta = 0$ (a case that we do not consider), this condition is known as the angular cutoff assumption, under which 
the collisional operator can be split into the gain and loss terms
\begin{align} \label{qsplit}
Q(f,f) = Q^+(f,f) - Q^-(f,f),
\end{align}
where
 \begin{align*}
& Q^+(f,f)(t, v) \, = \, \int_{\R^d} \int_{S^{d-1}} \, f'  f'_*  \, B(|u|, \hat{u}\cdot\sigma) \, d\sigma \, dv_*, \\ 
& Q^-(f,f)(t, v) \, = \, f(v) \, \int_{\R^d} \int_{S^{d-1}} \,  f_* \, B(|u|, \hat{u}\cdot\sigma) \, d\sigma \, dv_*.
\end{align*}
In 1963 Grad \cite{gr63} proposed considering a bounded angular kernel $b(\cos\theta)$ and pointed out  that different cutoff conditions could be implemented too. Since then the cutoff theory developed extensively, with the belief that removing the singularity of the angular kernel should not affect properties of the equation. Recently, however, it has been observed (see for example \cite{li94}, \cite{de95}, \cite{dego00}, \cite{dewe04}) that the singularity of $b(\cos\theta)$ carries regularizing properties.  This, in addition to the analytical challenge, motivated further study of the non-cutoff regime.

The typical non-cutoff assumption in the literature is the condition \eqref{eq-nc} with $\beta =2$.
However, we work in the non-cutoff regime where the parameter  $\beta \in (0,2]$ is allowed to vary and we will see how the strength of the singularity of $b$ influences our main result.
In this setting, the splitting \eqref{qsplit} is not valid, which is one of the technical challenges that non-cutoff setting brings. In order to address this obstacle 
we exploit angular cancellation properties (for details please see Section \ref{sec-aux}).

\begin{remark}
In the physically relevant case corresponding to the dimension $d=3$,  when forces between particles are governed by an inverse power law long range interaction potential $\phi(x) = C x^{-(p-1)}$,  $C>0$, $p>2$, the angular kernel $b(\cos{\theta})$ has been derived by H. Grad \cite{gr63}  (see also  \cite{ce88}) and is shown to have the following form
\begin{align}\label{intrapot-exp}
& b(\cos{\theta}) \; \sin{\theta} \sim C \; \theta^{-1 -\nu}, \qquad \theta \rightarrow 0^+, \nonumber\\
& \nu = \frac{2}{p-1}, \qquad  \gamma = \frac{p-5}{p-1}, \qquad p>2.
\end{align}
Note that this model satisfies \eqref{eq-nc} with any $r>\nu$.

 \end{remark}

\subsection*{Weak formulation of the collision operator $\boldsymbol{Q(f,f)}$}
Thanks to the symmetries associated to the collisional form $Q(f,f)$, defined in the strong form  \eqref{eq-q}, the collisional operator has a weak formulation that is very important for the analytical manipulation of the equation. Indeed,  for any test function $\phi(v)$, $v\in{\R^d}$, one has (see for example \cite{ce88})
\begin{align}\label{weak-Q}
	\int\limits_{\R^d} Q(f,f)(t,v) \, \phi(v) dv 
		&=  \frac{1}{2}\; \iint\limits_{\R^{2d}} f(v)\, f(v_*) \; G_{\phi}(v,v_*)\, dv_*\, dv, \\
	G_{\phi}(v,v_*) 
	&=\int_{S^{d-1}} \left(  \phi(v') + \phi(v'_*) -\phi(v) -\phi(v_*)\right) B(|u|, \hat{u}\cdot \sigma) \;  \; d\sigma.  \nonumber
\end{align}

The key aspect of the equation in the weak formulation is expressed in the weight $G_\phi$ as it carries all the information about collisions through the collisional kernel $B$, which is averaged over the unit sphere against test functions $\Delta \phi =  \phi(v') + \phi(v'_*) -\phi(v) -\phi(v_*)$.  Crucial estimates on the function $G_\phi$  referred to in the Boltzmann equation literature as  Povzner estimates are described below.

In the angular cutoff case,  positive and negative contributions are treated separately and such estimates  are used to estimate the positive part of $G_\phi$.  A sharp
form of  angular averaged Povzner estimates from \cite{bo97, bogapa04, gapavi09} is obtained for general test functions $\phi(v)$ which are positive and  convex.  They are crucial for the study of moments summability, the main point of this manuscript.

When $\phi(v)=(1+|v|^2)^{k/2} = \langle v\rangle^k$,  these estimates, originally developed by Povzner \cite{po62}, yield ordinary differential inequalities for moment estimates that lead to an existence theory and generation and propagation of moments as developed in Elmroth \cite{el83},  Desvillettes \cite{de93}  Wennberg \cite{we97} and   Mischler, Wennberg \cite{miwe99}. These estimates were also obtained in the  non-cutoff case  by Wennberg \cite{we96} for hard potentials. Uniqueness theory  to solutions of the Boltzmann equation for hard potentials was first developed by Di Blassio in \cite{dibl74}.

When the angular part of the collision kernel is not integrable, i.e. the non-cutoff case, one needs to expand $\Delta \phi$  in  terms of  $v'- v$ and $v'_*  - v_*$, since both  
are a multiples of  $|u|\, \sin{\theta/2}$. For this strategy to succeed, the spherical integration variable $\sigma  \in S^{d-1}$ must be decomposed as $\sigma = \hat u \cos \theta + \omega \sin \theta$, corresponding to the polar direction of the relative velocity $u$, and the azimuthal direction $\omega \in S^{d-1}$  satisfying $u\cdot \omega=0$. 
This decomposition also plays a fundamental role in our  derivation of the angular averaged Povzner with singularity cancellation in the proof of  
Lemma~\ref{averaged-pov}.

\begin{remark}
We note that the identity \eqref{weak-Q} can also be expressed in a double mixing (weighted) convolutional form (\cite{gata09, alcaga10, alga11})
\begin{align*}
 \int\limits_{\R^d} Q(f,f)(t,v) \, \phi(v) dv &=  \frac{1}{2}\; \iint\limits_{\R^{2d}} f(v)f(v-u)\; G_{\phi}(v,u)\, du\, dv  \\
G_{\phi}(v,u) &=\int_{S^{d-1}} \left(  \phi(v') + \phi(v'-u') -\phi(v) -\phi(v-u)\right) B(|u|, \hat{u}\cdot \sigma) \;  \; d\sigma  \nonumber
\end{align*}
since both $v'$ and $v'_*$ can be written as functions of  $v, u$ and $\sigma$ from \eqref{velo-pre},  and so the  weight function $G_\phi(v,u)$ is an average over $\sigma\in S^{d-1}$.
\end{remark}

\subsection{Moments of solutions to the Boltzmann equation} 
From the probabilistic viewpoint, moments of a probability distribution density $f(t,v)$ with respect to the variable $v$ are integrals of such density weighted by functions $\phi(v)$. These are important objects to study as they express average quantities that have significant meaning for the model under consideration. They are the so called observables. In this sense  polynomial moments correspond to such integrals for polynomial weights, and exponential moments are for exponential weights.

We now recall definitions of polynomial and exponential moments and we here introduce the Mittag-Leffler moments, which are a natural generalization of the exponential moments.

\begin{definition}[Polynomial and exponential moments]
 \label{def-mom}
Polynomial moment of order $q$  and exponential moment of order $s$ and rate $\alpha$ are respectively defined by:
\begin{align}
	& m_q(t) \, := \, \int_{\R^d} f(t,v) \;  \bra v \ket^q \; d(v), \label{def-polymom} \\
	&\mathcal{M}_{\alpha, s}(t) := \int_{\R^d}  f(t,v) \; e^{\alpha \, \bra v \ket^s} \; dv. \label{def-expmom}
\end{align}

\end{definition} 

\begin{remark} 
Using the Taylor series expansion, the exponential moment of order $s$ and rate $\alpha$ can also be written as the following sum
\eqn\label{def-expmomsum}
	\mathcal{M}_{\alpha, s}(t) \; = \;\sum_{q=0}^\infty \frac{m_{qs}(t) \; \alpha^q}{q!}.
\eeqn
\end{remark}

\begin{remark}
Polynomial moments can be expressed in terms of the norm of a natural Banach space in the context of the Boltzmann equation. Namely, if we denote
\begin{align*}
L^1_k = \{f \in L^1(\R^d): \int_{\R^d} f    \langle v\rangle^{k} dv =   \int_{\R^d} f    \left(1 + |v|^2 \right)^{k/2} dv <\infty\},
\end{align*}
then
\begin{align}
m_q(t) = \|f \|_{L^1_q}(t).
\end{align} 
Also, note that 
\begin{align}\label{moml1}
\|f\|_{L^1_q} \leq \|f\|_{L^1_{q'}}, \quad \mbox{for any} \,\, q \leq q'.
\end{align}
\end{remark}

Note that this expression is associated to the notion of $L^1$ exponential tail behavior described in \eqref{tt} and \eqref{tt1}. Consequently, finiteness of exponential moments can be understood as implying that the function $f(t,v)$ has an  exponential tail in $v$. In this paper, we study whether this property can be generated or propagated in time for the case of variable hard potentials in the non-cutoff case.

\smallskip

Because our summability estimates lead to expressions  similar to that of \eqref{def-expmomsum}, yet  having  $\Gamma (aq + 1)$  as a generalization of factorials with non-integer $a>1$,  we are  motivated  to use Mittag-Leffler functions, as they are conceived as a generalization of the Taylor expansion of the exponential function. More precisely, for a parameter $a>0$, Mittag-Leffler function is defined via
\eqn\label{def-ml}
	\cE_a (x) := \sum_{q=0}^\infty \frac{x^q}{\Gamma(aq+1)}.\label{eq-mlfc}
\eeqn
Note that for $a=1$, the Mittag-Leffler function  coincides with the Taylor expansion of the classical exponential function $e^x$. It is also well known (see e.g. \cite{emot53}, page 208.) that for any $a>0$, the Mittag-Leffler function asymptotically behaves like an exponential function of order $1/a$,   that is
\eqnn
 \cE_a (x)  \sim  e^{x^{1/a}}, \qquad \mbox{as} \,  x \rightarrow \infty.
\eeqnn

Since $\bra v \ket^2$ is the building block for our calculations, we prefer to have $x^2$ as the argument of Mittag-Leffler function when generalizing $ e^{\alpha \, x^s}$
\begin{align}\label{eq-asym}
 \cE_{2/s} (\alpha^{2/s} \, x^2) \sim  e^{\alpha \, x^s}, \qquad \mbox{for} \,  x \rightarrow \infty.
\end{align}

Hence, they satisfy the following, with some positive constants $c, C$
\begin{align}\label{ml-asymp}
c \; e^{\alpha x^s} \; \le \; \cE_{2/s} (\alpha^{2/s} x^2) \; \le \; C\; e^{\alpha x^s}.
\end{align}

This motivates our definition of Mittag-Leffler moments.
\vspace{5pt}
\begin{definition}[Mittag-Leffler moment] \label{def-mom-ML} 
 Mittag-Leffler moment of  order $s$ and rate $\alpha>0$ of a function $f$ is introduced via
\eqn
	\displaystyle\int_{\R^d}  f(t,v) \;\; \cE_{2/s} (\alpha^{2/s} \, \bra v \ket^2 ) \; dv.
\eeqn
\end{definition} 

\begin{remark} 
In the rest of the paper we will  use the fact that Mittag-Leffler moments can be represented as the following sum (a time series form), which follows from \eqref{eq-mlfc}
\begin{equation}\label{def-mlmomsum}
			\displaystyle\int_{\R^d}  f(t,v) \;\; \cE_{2/s} (\alpha^{2/s} \, \bra v \ket^2 ) \; dv
		\;\; = \;\; \displaystyle \sum_{q=0}^\infty \frac{m_{2q}(t) \, \alpha^{2q/s}}{\Gamma(\frac{2}{s} \, q+1)}.	
\end{equation}
\end{remark}

\begin{remark}
Formally, by taking $k = \frac{2q}{s}$, the above sum becomes 
\eqnn
\sum_{k \in \frac{2}{s} \Z} \frac{m_{ks}(t) \;  \alpha^k}{\Gamma(k+1)},
\eeqnn
that we show it relates to  the time series in \eqref{tt1} with the difference being that the summation here goes over the fractions.

\end{remark}

\subsection{The main results}

There are two important results in this manuscript. The first one relates to the angular averaged Povzner estimate with cancellation. It gives an estimate of the weight function $G_\phi$ in the weak formulation \eqref{weak-Q} when the test function is a monomial $\phi(v) = \langle v \rangle^{rq}$. We denote this weight function by
\begin{align}
G_{rq} := G_{\langle v \rangle^{rq}} :=  \int_{S^{d-1}} \left(  \langle v' \rangle^{rq}  + \langle v'_* \rangle^{rq} 
	-\langle v \rangle^{rq} - \langle v_* \rangle^{rq}\right) B(|u|, \hat{u}\cdot \sigma) \;  \; d\sigma 
\end{align}

\begin{lemma}\label{averaged-pov}
Suppose that the angular kernel $b(\cos\theta)$ satisfies the non-cutoff condition \eqref{eq-nc} with $\beta =2$. Let $r, q>0$. Then the weight function satisfies
\begin{align} \label{part1}
\nonumber
G_{rq}(v, v_*) & \leq |v-v_*|^\gamma \, \left[ - \, A_2 \, \Big( \langle v \rangle^{rq} +  \langle v_* \rangle^{rq}\Big)
		    \quad + \, A_2 \, \Big( \langle v \rangle^{rq-2} \langle v_* \rangle^{2} 
		    + \langle v \rangle^{2} \langle v_* \rangle^{rq-2}\Big) \right.	 \\  
	& \left. \quad+  \, \varepsilon_{qr/2} \, A_2 \, \frac{qr}{2} \left( \frac{qr}{2}-1\right)\,  \langle v \rangle^2 \langle v_* \rangle^2 \,
	    \Big( \langle v \rangle^2 + \langle v_* \rangle^2 \Big)^{ \frac{qr}{2}-2 }\right],
\end{align}
where $ A_2 = |S^{d-2}| \int_0^\pi \,  b(\cos{\theta})  \, \sin^d {\theta} \, d\theta$ is finite by \eqref{eq-nc}.
The sequence  $\varepsilon_{qr/2}=: \varepsilon_{\bf q}$, defined as
\begin{equation} \label{def-eps}
	\varepsilon_{\bf q} := \frac{2}{A_2} |S^{N-2}| \displaystyle \int_0^\pi
      		\left( \int_0^1 t \left( 1 - \frac{\sin^2 \theta}{2} t \right)^{{\bf q}-2} dt\right) b(\cos \theta) \; \sin^N \theta \; d\theta,
\end{equation}
has the following decay properties.  If $b(\cos{\theta})$ satisfies the non-cutoff assumption \eqref{eq-nc} with $\beta \in (0,2]$, then
\begin{equation}\label{eq-epsdecay}
0\ <\ \varepsilon_{\bf q} \; {\bf q}^{1 - \frac{\beta}{2}} \  \rightarrow 0, \ \quad  \mbox{as} \; {\bf q} \rightarrow \infty.
\end{equation}
\end{lemma}

The sequence $\varepsilon_{\bf q}$  is the same as in \cite{lumo12}. Its decay properties \eqref{eq-epsdecay} are also proved in \cite{lumo12}, after invoking  angular averaging and the dominated convergence theorem. Condition \eqref{eq-epsdecay} is crucial for finding the highest order $s$ of Mittag-Leffler moment that can be propagated in time.

\smallskip
\begin{remark}
This lemma relies on the polynomial inequality presented in Lemma \ref{lem-polyII}.
The decay rate of $\varepsilon_{\bf q}$ is fundamental for the success of summability arguments, yet is not relevant for the generation and propagation of polynomial moments. In the angular cutoff case when term-by-term techniques were used, the corresponding constant had a rate $\varepsilon_q \approx q^{-r}$, with $r$ depending on the integrability of $b$, see \cite{bo97, bogapa04, gapavi09}. When the partial sum technique was employed in \cite{alcagamo13}, the precise rate was not needed any longer.  Here however, in the non-cutoff case, the knowledge of the precise decay rate of $\varepsilon_{\bf q}$ becomes important again because of extra power of $q$ in the last term of the right-hand side of \eqref{averaged-pov}.
\end{remark}

\medskip

The second main result, presented as an a priori estimate, consists of two parts. First, under the non-cutoff assumption \eqref{eq-nc} with $\beta=2$, we provide a new proof of the generation of exponential moments of order $ s \in (0, \gamma]$. Second, we show the propagation in time of the Mittag-Leffler moments of order $s \in (\gamma, 2)$. When $ s \in (\gamma, 1]$, $\beta=2$ in the non-cutoff \eqref{eq-nc} is assumed. When $s \in (1,2)$, the angular kernel 
is assumed to be less singular. 
Before we state the theorem, we remind the reader of the following notation
\begin{align*}
L^1_k = \{f \in L^1(\R^d): \int_{\R^d} f    \langle v\rangle^{k} dv <\infty\}.
\end{align*}
This is the natural Banach norm to solve the Boltzmann equation. 

\begin{theorem}[Generation and Propagation of Exponential-like moments]\label{thm}
Suppose $f$ is a solution to the Boltzmann equation \eqref{eq-be} with the collision kernel of the form \eqref{eq-kernel} for hard potentials \eqref{eq-gamma}, and with initial data $f_0 \in L^1_{2}$.
\vspace{-0,1in}
\begin{itemize}
\item[(a)] (Generation of exponential moments) If  the angular kernel satisfies the non-cutoff condition \eqref{eq-nc} with $\beta = 2$, then the exponential moment of order $\gamma$ is generated with a rate $r(t)={\alpha \, \min\{t,1\}}$. More precisely, there are positive constants $C, \alpha$, depending only on $b$, $\gamma$ and initial mass and energy, such that
\begin{equation}\label{eq-gener}
	\int_{\R^d}  \, f(t,v) \; e^{\alpha \, \min\{t,1\} \, |v|^\gamma} \, dv \, \leq \, C, \quad \mbox{for} \;\; t\geq 0.
\end{equation}
\item[(b)] (Propagation of Mittag-Leffler moments) Let $s \in (0, 2)$ and suppose that the Mittag-Leffler moment of order $s$ of the initial data $f_0$ is finite with a rate $r=\alpha_0$, that is,
\begin{equation}\label{eq-id}
	\displaystyle\int_{\R^d}  f_0(v) \;\; \cE_{2/s} (\alpha_0^{2/s} \, \bra v \ket^2 ) \; dv < M_0. 
\end{equation}
Suppose also that the angular cross-section satisfies assumption  \eqref{eq-nc}
\begin{align}\label{s_condition}
	& \mbox{with} \;\; \beta=2,    &  \mbox{if} \;\;s\in (0, 1] \nonumber \\
	&\mbox{with} \;\; \beta = \frac{4}{s}-2,   & \mbox{if} \;\; s\in (1, 2).
\end{align}
Then, there exist positive constants $C, \alpha$, depending only on $M_0$, $\alpha_0$, $b$, $\gamma$ and initial mass and energy such that the Mittag-Leffler moment of order $s$ and rate $r(t) =\alpha$ remains uniformly bounded in time, that is 
\begin{equation}\label{eq-prop}
	\displaystyle\int_{\R^d}  f(t,v) \;\; \cE_{2/s} (\alpha^{2/s} \, \bra v \ket^2 ) \; dv < C, \quad \mbox{for} \; \; t\geq 0.
\end{equation}
\end{itemize}
\end{theorem}

\begin{remark}\label{mainrmk}
The angular singularity condition $\beta = \frac{4}{s}-2$ in the case of Mittag-Leffler moments of order $s \in (1,2)$, continuously changes from $\beta =2$ (for $s=1$) to $\beta = 0$ (for $s=2$). Hence condition $\beta = \frac{4}{s}-2$  continuously interpolates between the most singular kernel typically considered in the literature, which is  \eqref{eq-nc} with $\beta =2$, and  an angular cutoff condition, which corresponds to \eqref{eq-nc} with $\beta=0$. This also tells us that in the most singular case one can propagate exponential moments of order $s\le1$, while in angular cutoff cases one can propagate exponential moments of order $s\le2$ (to be completely rigorous, Theorem \ref{thm} goes up to $\beta >0$ i.e. $s<2$, but \cite{alcagamo13} already established the case $\beta =0$ i.e. $s=2$). In other words, the less singular the angular kernel is, the higher order exponential moment propagate in time.\end{remark}

\begin{remark}
The propagation result of the theorem can be interpreted in two ways. First, for a Mittag-Leffler (or exponential) moment of order $s$ to be propagated, the singularity of $b$ should be such that it satisfies \eqref{eq-nc} with $\beta= \frac{4}{s} - 2$.  On the other hand, given an angular kernel  $b$ that satisfies condition \eqref{eq-nc} with a parameter $\beta \in (0,2]$, one can propagate Mittag-Leffler (and exponential) moments of order $s\le \frac{4}{\beta+2}$.
\end{remark}

\begin{remark}
We note two types of solutions that can be used in the previous theorem. One example are weak solutions, whose existence was proven by Arkeryd \cite{ar81} and later extended by Villani \cite{vi98}, under the assumption that initial data has finite mass, energy, entropy and a moment of order $2+\delta$, for any $\delta>0$. Another type of solutions that could be used are measure weak solutions constructed by Lu and Mouhot \cite{lumo12} (see also the result of Morimoto, Wang and Yang \cite{mowaya16}). These solutions exist if initial mass and energy are finite, provided that the angular kernel satisfies the following condition $\int_0^\pi b(\cos\theta) \sin^d\theta (1+ |\log(\sin\theta)|) < \infty$ , which automatically holds for kernels that satisfy condition \eqref{eq-nc} with $\beta<2$.
\end{remark}

\begin{remark}
Thanks to the fact \eqref{ml-asymp} that Mittag-Leffler functions asymptotically behave like exponential functions, finiteness of exponential moment of order $s$ is equivalent to finiteness of the corresponding Mittag-Leffler moment. This, in turn implies, as a corollary of Theorem \ref{thm} (b), the propagation of classical exponential moments.
\end{remark}

\begin{remark}
In the case of inverse-power law model described via \eqref{intrapot-exp}, in which  hard potentials correspond to $p>5$, the non-cutoff condition \eqref{eq-nc} is satisfied for $\beta>\nu$.  Hence, Mittag-Leffler moments of orders $s <  2 - \frac{2}{p}$ can be propagated in time. In the graph below $y$-axis represents the order of exponential tails. The dashed red line marks the highest order of exponential moments that can be generated, while the blue line marks the highest order of Mittag-Leffler moments that can be propagated in time. This graph visually confirms that our propagation result indeed goes beyond the rate of potentials $\gamma$.

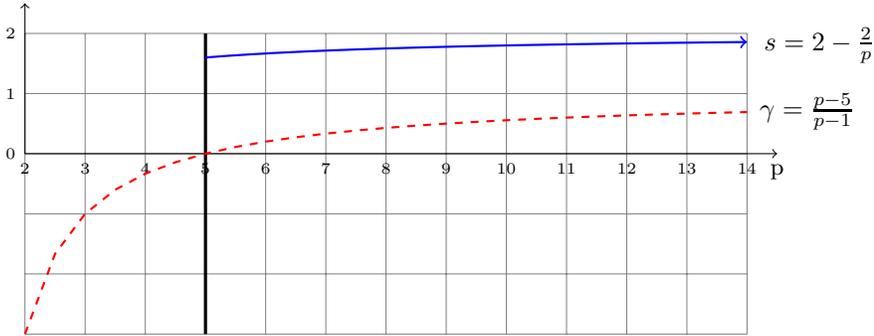
\tikz[scale=0.8]{
\node at (15,.7) {$\gamma = \frac{p-5}{p-1}$};
\node at (15.2,1.8) {$s=2 - \frac{2}{p}$};
\node at (14.5, -0.3) {\small  p};
\draw [help lines,xstep=1cm, ystep=1cm] (2,-3) grid (14,2);
\foreach \x in {0,...,2}{\node [left] at (2,\x) {\tiny \x};}
\foreach \y in {2,...,14}{\node [below] at (\y,0) {\tiny \y};}
\draw [->] (2,0) -- (14.5,0) ;
\draw [->] (2,0) -- (2,2.5) ;
\draw[very thick, -] (5,-3)--(5,2);
\draw [red,thick,dashed] (2,0) plot [domain=2:14] (\x,{(\x-5)/(\x-1)});
\draw [blue,thick,->] (2,0) plot [domain=5:14] (\x, {2 - 2/(\x)});
}

\end{remark}

\subsection{A strategy for proving Theorem \ref{thm}} \label{sub-strategy}
 Details are provided in Section \ref{sec-ml} and Section 6. The proof is inspired by the recent work \cite{alcagamo13}, where propagation and generation of tail behavior \eqref{tt1} is obtained for angular cutoff regimes.
 
 Our goal is to prove that solutions $f(t,v)$ of the Boltzmann equation for hard potentials and angular non-cutoff conditions admit $L^1$-Mittag-Leffler moments  with parameters $a = \tfrac{2}{s}$ and  $\alpha(t) = r(t)$ to be found. Because of the asymptotic behavior \eqref{eq-asym}, that would imply that asymptotic limit for large values of $v$ is, indeed, and exponential tail in $v$-space, with order $s$ and rate $r(t)=\alpha(t)$. Thus, our proof is based on studying partial sums of Mittag-Leffler functions $\cE_a(\alpha^a x^2)$, with parameter $a = \tfrac{2}{s}$ and with rate $\alpha(t)$.

To this end, we work with  $n$-th partial sums associated to Mittag-Leffler functions, defined as
\begin{equation}\label{prop-E}
 	\cE^n_a (\alpha, t) = \displaystyle \sum_{q=0}^n \frac{m_{2 q}(t) \; \alpha^{aq}}{\Gamma(aq+1)}. 
 		\qquad 
  \end{equation}
We need to prove that there exits a positive rate $\alpha(t)$ and a positive parameter $a$, both uniform in $n$,  such the sequence of finite sums converges as $n\to \infty$. In particular, we need show that $\cE^n_a (\alpha, t)$ is bounded by a constant independent of time and independent of $n$. The values for $a, \alpha$ and the bound of the partial sums are found and shown  to depend on data parameters given by the collisional kernel characterization and properties of the initial data.
 
 In order to achieve all of this,  we derive a differential inequality for $\cE^n_a=\cE^n_a(\alpha, t)$. The first step in this direction is to obtain differential inequalities for moments $m_{2 q}(t)$, by studying the balance  
\begin{align}\label{multiply}
 m'_{2q}(t) = \int_{\R^d} Q(f, f) (t,v) \; \bra v \ket^{2q} dv.
\end{align}
that is a consequence of the Boltzmann equation. The right hand side is estimated by bounding the polynomial moments of the collision operator by non-linear forms of moments $m_{k}(t)$ of order up $k=2q+\gamma$, with  $0<\gamma\le 1$.
This requires finding the estimates of the weak formulation \eqref{weak-Q} with test functions $\phi(v) = \langle v \rangle^k$. Consequently, we need to estimate the angular integration within the weight function $G_{\langle v \rangle^{2q}}(v,v')$
\begin{align}
 \int_{\mathcal{S}^{d-1}} \left(\bra v' \ket^{2q} + \bra v'_* \ket^{2q} - \bra v \ket^{2q} - \bra v_* \ket^{2q} \right) b(\cos\theta) d\sigma.
\end{align}
These estimates will lead, thanks to \eqref{multiply} and \eqref{weak-Q}, to the following differential inequality for polynomial moments
 \begin{equation}\label{str-ode-m}
 \begin{array}{lcl}
    \quad\qquad  m'_{2q}  
	  &\leq&  -K_1 \, m_{2q+\gamma} \, + \,  K_2 \, m_{2q}  \\
	  
	  && + \, K_3 \, \varepsilon_q \, q \, (q -1) 
	      \displaystyle \sum_{k=1}^{k_q} \binom{q-2}{k-1} 
	      \left(m_{2k+\gamma} \, m_{2(q-k)} + m_{2k} \, m_{2(q-k) +\gamma}\right).
  \end{array}	      
 \end{equation}
 where $K_1=A_2 C_\gamma$, where  $A_2$ was defined  in \eqref{part1} and $C_\gamma$ just depends on the rate of potentials $\gamma$. Similarly $K_2$ and $K_3$ depend on these data parameters as well. 
The key property of this inequality is that the highest order moment of the right-hand side comes with a negative sign which  is crucial for moments propagation and generation.  
Another important aspect of this differential inequality is the presence of the factor $q(q-1)$ in the last term, which was absent in angular cutoff cases. Because of it, it will be of great importance to know the decay rate for $\varepsilon_q$.

The second step (Section 4) consists in the derivation of a differential inequality for partial sums $\cE^n_a=\cE^n_a(\alpha, t)$ obtained by adding $n$ inequalities corresponding to \eqref{str-ode-m} for renormalized polynomial moments   $m_{2 q}(t)  {\alpha^{aq}}/{\Gamma(aq+1)}$. This will yield
\begin{equation} \label{str-Eode1} 
 \displaystyle 
 	\frac{d }{dt} \cE^n_a 
	\leq c_{q_0} \, + \, \Big( - K_1 \, \cI^n_{a,\gamma} \, + \,  K_1 \,c_{q_0} \, + \,  K_2 \,  \cE^n_a \, + \,  \varepsilon_{q_0} \; q_0^{2-a} \, K_3 \, C \,\cE^n_a \,  \cI^n_{a,\gamma} \Big)\, .
\end{equation} 
In particular we obtain an ordinary differential inequality for the partial sum $\cE^n_a$ that depends on a shifted partial sum $\cI^n_{a,\gamma}$, defined by
\begin{equation}\label{prop-I}
 	\cI^n_{a,\gamma} (\alpha, t) = \displaystyle \sum_{q=0}^n \frac{m_{2 q +\gamma}(t) \; \alpha^{aq}}{\Gamma(aq+1)}.
 \end{equation}
The derivation of the last term in the right hand side of \eqref{str-Eode1} requires a decay property of combinatoric sums of Beta functions. These estimates are very delicate and are presented in detail in Lemma~\ref {lem-sumB1}  and   Lemma~\ref {lem-sumB2} in the Appendix. The constants $K_1, K_2$ and $K_3$ only depend on the  singularity conditions \eqref{eq-nc}, and so  they are independent of $n$ and on any moment $q$.  The constant $c_{q_{0}}$ depends only on a finite number $q_0$ of  moments of the initial data. The choice of  $q_0$ is crucial to control the long time behavior of solutions to inequality \eqref{str-Eode1}, and it
 is done
  such that $\varepsilon_{q_0} \; q_0^{2-a}  \, K_3 < K_1/2 $, after using condition \eqref{eq-epsdecay}   in Lemma~\ref{averaged-pov}.
 
 Finally, after showing that   $\cI^n_{a,\gamma}(\alpha,t)$ is bounded below by sum of two terms depending linearly  on $\cE^n_a(\alpha, t)$ and on mass $m_0$,   and nonlinearly on  the rate $\alpha$,  we obtain the following differential inequality for partial sums in the case of propagation of initial Mittag-Leffler moments
\begin{equation*}
	\frac{d}{dt} \cE^n_a(t) \, \leq \,  - \,\frac{K_1}{2\, \alpha^{\frac{\gamma}{2}}} \cE^n_a (t)
		 \, + \, \frac{K_1 \, m_0 \, e^{\alpha^{1-a}}}{ 2 \alpha^{\frac{\gamma}{2}}}
		\, +   \cK_{0}  \ \  \mbox{\sl (Propagation estimate)}.
\end{equation*}
The constant  $\cK_0$ depends on data parameters characterizing $q_0, c_{q_{0}}$ and   $ K_i$, $i=1,2,3.$. In addition, 
 for the generation case, 
we obtain 
\begin{equation*}
 \frac{d}{dt} E^n_{\gamma} \leq -\frac{1}{t} \left(\frac{K_1(E^n_\gamma -m_0)}{2 \alpha} - C_{q_0}\right)  + \cK_0 \ \ \  \mbox{\sl (Generation estimate)}.
\end{equation*}
Thus,  the differential inequalities \eqref{str-Eode1} are reduced  to  linear ones.
Both inequalities have corresponding solutions for choices on parameters $ a$ and $\alpha$ that are independent on $n$ and time $t$, and will depend on $q_0$, which depends only on data parameters.

\section{Angular averaging lemma} \label{sec-aux}

\noindent  This section is about  the proof of  the angular averaging with cancellation, i.e. Lemma~\ref{averaged-pov}, a crucial step for controlling moments and  summability of their renormalization by the Gamma function.
 One of the tools used in the proof is the following estimate on symmetrized  convex binomial expansions.

\begin{lemma} \label{lem-polyII} {\bf [Symmetrized  convex binomial expansions estimate]}
\noindent Let $a, b \geq 0$, $t\in [0,1]$ and $p \in (0,1] \cup [2,\infty)$. Then
 \begin{align}
    \Big( t a  +  (1-t) b \Big)^p \, & + \, \Big( (1-t) a  + t b \Big)^p \, - \, a^p \, - \, b^p  \nonumber \\ 
     &  \leq \,- \, 2 t (1-t) \Big(a^p + b^p \Big) \, + \, 2 t (1-t) \Big( a b^{p-1}  + a^{p-1} b \Big). \label{poly}
\end{align}
\end{lemma}
\begin{prf}
Suppose $p \geq 2$. The case $p \in (0,1]$ can be done analogously. Due to the symmetry of the inequality (\ref{poly}), we may without the loss of generality assume that $a \geq b$.  Since all the terms have homogeneity $p$,  the inequality (\ref{poly}) is equivalent  to showing
\eqnn
F(z) \geq 0, \qquad \forall z \geq 1,
\eeqnn 
where $F(z)$ is defined by
 \begin{equation*}
  F(z) :=  \Big(1- 2t(1-t) \Big) \big( z^p +1 \big) \; + \; 2t(1-t) \big(z + z^{p-1} \big) 
 \; - \; \Big(t z + (1-t) \Big)^p \; - \; \Big( (1-t) z + t \Big)^p.
 \end{equation*} 
 It is easy to check that
\begin{align*}
 	 F''(z) = (p-1) \Bigg[p & \Big(1- 2t(1-t) \Big) z^{p-2} \; + \;   2t(1-t) (p-2) z^{p-3} \\
	  & \; - \; p t^2 \Big(t z + (1-t) \Big)^{p-2} \;  - \;   p(1-t)^2 \Big( (1-t) z + t \Big)^{p-2} \Bigg].
\end{align*}
 As $tz + (1-t)$ and $(1-t)z + t$ are two convex combinations of $z$ and $1$, and since $z\geq 1$, we have that $tz + (1-t) \leq z$ and $(1-t)z + t \leq z$. Since $p\geq 2$, this implies $(tz + (1-t))^{p-2} \leq z^{p-2}$ and $((1-t)z + t)^{p-2} \leq z^{p-2}$. Therefore,
\begin{align*}
   \frac{F''(z)}{p-1} & \geq  p \big(1- 2t(1-t) \big) z^{p-2} \; + \;  2t(1-t) (p-2) z^{p-3} - \; p t^2 z^{p-2} \;-\; p(1-t)^2 z^{p-2} \\
    &  =   2t(1-t) (p-2) z^{p-3} \\
 & \geq 0.
\end{align*}
Thus, $F''(z) \geq 0$ for $z \geq 1$. So, $F'(z)$ is increasing. Since $ F'(1) = 0$, we have that $F'(z) \geq 0$ for $z\geq 1$. Finally using the fact that $F(1) = 0$, we conclude $F(z) \geq 0$ for $z \geq 1$.  
\end{prf}

We are now ready to prove the new form of the angular averaging with cancellation type lemma. For another version, see \cite{lumo12}.

{\it Proof of Lemma~\ref{averaged-pov}} Recall the definition of the weight $G_{rq}$
\begin{align}\label{weak-qr}
	& G_{rq}(v, v_*) := |v-v_*|^\gamma \int_{\mathcal{S}^{d-1}} b(\cos\theta) \, \sin^{d-2}{\theta}
		\;  \Delta\langle v \rangle^{rq} \; d\sigma, \\
	& \mbox{where} \;\; \Delta\langle v \rangle^{rq} = \langle v'\rangle^{rq} + \langle v'_*\rangle^{rq} - \langle v\rangle^{rq} - 
		\langle v_*\rangle^{rq}. \nonumber
\end{align}

\begin{wrapfigure}{R}{0.35 \textwidth}
\setlength{\unitlength}{1cm}
\begin{tikzpicture}[scale=0.45]
\centering
\begin{scope}
   \draw (0,0) circle (4);
  \draw (-4,0) arc [radius = 8, start angle=-120, end angle= -60];
  \draw [dashed] (-4,0) arc [radius = 8, start angle=120, end angle= 60];
  
  \draw [->] (0,0) -- (0,4);
   \node at (0.25,3.6) {$\boldsymbol{\hat{u}}$};
  
  \draw [->] (0,0) -- (2.65,2.15);
  \node at (2.3, 2.1) {$\boldsymbol{\sigma}$};
  \draw [dotted] [thick] (2.65, 2.15) -- (2.65, -0.4);
  
  \draw [->] (0,0) -- (3.1,-.45);
  \node at (3.2, -.7) {$\boldsymbol{\omega}$};
  \draw [dotted] (0,4) arc [x radius = 3.1cm, y radius = 4cm, start angle = 90, end angle = -7];

  \draw [->] (0,0) -- (-3,-.5);
  \node at (-3, -.8) {$\boldsymbol{j}$};
   \draw [dotted] (-3, -.5) arc [x radius = 2.7cm, y radius = 4cm, start angle = 187, end angle = 90];

   \draw [->] (0,0) -- (-1.7, 3.4);
   \node at (-1.2, 3.3) {$\boldsymbol{\hat{V}}$};
   \draw [dotted] (-1.7, 3.4) -- (-1.7, -0.3);
   
   \draw (0,0.8) arc [radius=0.8, start angle=90, end angle = 40];
   \node at (0.2, 0.4) {$\boldsymbol{\theta}$};
   
   \node at (1,-1.4) {$S^{d-2}$};
   \node at (3.5,-3) {$S^{d-1}$};
\end{scope}
\end{tikzpicture}
\caption [] {\hskip -0.1in \begin{minipage}[c]{1.\linewidth} Decomposition of $\sigma$. \end{minipage}}
\vspace{-0.3in}
\end{wrapfigure}

 This integral is rigorous even in cases when $\int_{S^{d-1}}  B(|u|, \cos{\theta}) \; d\sigma$ is unbounded, by an angular cancellation. A natural way of handling the cancellation is to decompose $\sigma \in S^{d-1}$ into $\theta \in [0,\pi]$ and its corresponding azimuthal variable $\omega \in S^{d-2}$, i.e. 
$\sigma = \cos{\theta}\; \hat{u} + \sin{\theta} \; \omega, $    where $S^{d-2}(\hat{u}) = \{ \omega \in S^{d-1} : \omega \cdot \hat{u} = 0 \}$. See Figure 2.

This decomposition allows handling the lack of integrability concentrated at the origin of the polar direction $\theta=0$. However, it requires a specific way of decomposing
$\bra v' \ket^2$ and $\bra v'_* \ket^2$ that separates the part that depends on $\omega$. More precisely, $\bra v' \ket^2$ and $\bra v'_* \ket^2$  are decomposed into a sum of a convex combination of the local energies proportional to a function of the polar angle $\theta$, and another term depending on both the polar angle and $\omega$ (see the Appendix for details)
\begin{flalign}\label{ep}
&\bra v' \ket^2 = E_{v, v_*}(\theta) \; + \; P(\theta, \omega), \\  \nonumber
&\bra v'_* \ket^2 = E_{v,v_*} (\pi - \theta) \; - \; P(\theta, \omega).&
\end{flalign}
Here $P(\theta, \omega)= |v\times v_*| \, \sin\theta \, (j\cdot \omega)$ is a null form in $\omega$ by averaging, i.e. $$\int\limits_{S^{d-2}} P(\theta, \omega) d\omega=0,$$
and $E_{v,v_*}(\theta)$ is a convex combination of $\bra v \ket^2$ and $\bra v'_* \ket^2$ given by
\begin{align*}
& E_{v,v_*}(\theta) \; = \; t \;  \bra v \ket^2  \; + \; (1-t) \; \bra v_* \ket^2, \qquad
 \mbox{where} \quad t \; = \; \sin^2{\frac{\theta}{2}}.
\end{align*}
These two fundamental properties make the weight function $G_{rq}(v,v_*)$ well defined for every $v$ and $v_*$ for sufficiently smooth test functions ($\phi \in C^2(\R^d)$) even under the  non-cutoff assumption \eqref{eq-nc} with $\beta=2$. In fact,  Taylor expansions associated to   $\langle v' \rangle^{rq}$ are a sum of a power of  $E_{v,v_*}(\theta)$, plus a null form in the azimuthal direction, plus a residue proportional to $\sin^2{\theta}$ that will secure the integrability of the angular cross section with respect to the scattering angle $\theta$.
Indeed, Taylor expand $\langle v' \rangle^{rq}$ around $E(\theta)$ up to the second order to obtain 
\begin{align}
 \langle v' \rangle^{rq} \;
	&  = \; \Big( E_{v,v_*}(\theta) \, + \, h \sin(\theta) \, (j\cdot \omega) \Big)^{\frac{rq}{2}} \\ \nonumber
    	& = \; \big(E_{v,v_*}(\theta)\big)^{rq/2}  \, 
		+ \, \frac{rq}{2} \, \big(E_{v,v_*}(\theta)\big)^{\frac{rq}{2}-1} \, h \, \sin{\theta} \, (j \cdot \omega) \\ 	\nonumber	 
	& \qquad +  \frac{rq}{2} \, \left(\frac{rq}{2}-1\right) \, h^2 \, \sin^2{\theta} \, (j \cdot \omega)^2 
	  	\displaystyle \int_0^1 (1-t)  \big[E(\theta) 
		+ t \, h \, \sin{\theta} \, (j \cdot \omega)\big]^{\frac{rq}{2}-2} \; dt.   	\label{eq-taylor}
\end{align}
Similar identity can be obtained for $\bra v'_* \ket^{rq}$.

 Since the collisional cross section is independent of the azimuthal integration we will make use of the following property.  Any vector $j$ laying in the plane orthogonal to the direction of $u$ is nullified by multiplication and averaging with respect to  the azimuthal direction with respect to $u$, that is    $\int_{S^{d-2}} j \cdot \omega \; d\omega = 0$. 
 
 Therefore,  we can write $G_{rq}(v,v_*)$ as the sum of two integrals on the $S^{d-1}$ sphere, whose first integrand contains the zero-order order term  of the Taylor expansion of, both,  
 $\bra v'_* \ket^{rq}$ and  $\bra v'\ket^{rq}$ subtracted by their corresponding un-primed forms, while the second integrand is just the second order term of the Taylor expansion \eqref{eq-taylor}
\begin{align}
& G_{rq}(v,v_*)  = I_{1} \; + \; I_{2} \\ \nonumber
& = \displaystyle \int_0^\pi \displaystyle \int_{S^{d-2}}
	         \Big( E_{v,v_*}(\theta)^{rq/2} + E_{v,v_*}(\pi - \theta)^{rq/2}
		    - \langle v \rangle^{rq} - \langle v_*\rangle^{rq} \Big) \; 
		      b(\cos{\theta}) \;\sin^{d-2}{\theta} \; d \omega \; d\theta   \\ \nonumber
&  \qquad+ \; \frac{rq}{2} \left(\frac{rq}{2}-1\right)  \, h^2 \, \displaystyle \int_0^\pi \sin^d{\theta} \; b(\cos{\theta})
	        \displaystyle \int_{S^{d-2}}   (j \cdot \omega)^2
		 \displaystyle \int_0^1 (1-t) \\ \nonumber
& \quad\qquad    \Big(  \big[E_{v,v_*}(\theta) \, + \, t \, h \, \sin{\theta} \, (j \cdot \omega)\big]^{\frac{rq}{2}-2} 
		      + \big[E_{v,v_*}(\pi - \theta) \, -\, t \, h \, \sin{\theta} \big]^{\frac{rq}{2}-2} \Big)  \; dt d\omega d\theta.  \label{split}
\end{align}

 \noindent At this point we use inequality (\ref{poly}) to estimate the first integral $I_{1}$. We use it with $a=\langle v \rangle^2$,  $b=\langle v_* \rangle^2$ and $t=\cos^2{\frac{\theta}{2}}$, which yields
\begin{equation}\label{Izl}
\begin{array}{lcl}
	I_{1} &\leq& \big|S^{d-2}\big| \displaystyle \int_0^\pi 
	  -\frac{\sin^2{\theta}}{2}  \; 
		\Big( \langle v \rangle^{rq} + \langle v_* \rangle^{rq} \Big)
		 \, b(\cos{\theta}) \; \sin^{d-2}{\theta}\; d\theta\\

         &&\quad \qquad	 \; + \; \displaystyle \int_0^\pi  \frac{\sin^2{\theta}}{2}  \; 
		\Big( \langle v \rangle^{rq-2} \langle v_* \rangle^{2} +
                  \langle v \rangle^{2} \langle v_* \rangle^{rq-2}\Big)
	 \, b(\cos{\theta}) \; \sin^{d-2}{\theta}\; d\theta \\     
                  
  &=&  - A_2 \Big( \langle v \rangle^{rq} +  \langle v_* \rangle^{rq}\Big)
		    \; + \; A_2  \Big( \langle v \rangle^{rq-2} \langle v_* \rangle^{2} 
		    + \langle v \rangle^{2} \langle v_* \rangle^{rq-2}\Big).    
 \end{array} 
 \end{equation}
The constant $A_2$ was defined after \eqref{part1}.

For the second order term $I_{2}$, we use that $(j \cdot \omega)^2 \leq 1$ and $h=|v \times v_*| \leq \bra v\ket \, \bra v_*\ket $, and that (see \cite{lumo12})
 \begin{equation}
  \big|E_{v,v_*}(\theta) + t h \sin{\theta} \; (j \cdot \omega)\big| \leq 
      \Big( \langle v \rangle^2 + \langle v_* \rangle^2 \Big) \left( 1 - \frac{t}{4} \sin^2{\theta} \right),
 \end{equation}
to conclude
\begin{align*}
 I_{2}   \leq   \frac{rq}{2} \left(\frac{rq}{2}-1\right)  \, &\langle v \rangle^2 \langle v_* \rangle^2\, \big|S^{d-2}\big| 	 \displaystyle \int_0^\pi 	\sin^d{\theta} \; b(\cos{\theta}) \\		
	& \displaystyle \qquad\int_0^1
		2 (1-t) \left( \langle v \rangle^2 + \langle v_* \rangle^2 \right)^{\frac{rq}{2}-2} \;
		\left( 1 - \frac{1-t}{4} \sin^2{\theta} \right)^{\frac{rq}{2}-2} \; dt \; d\theta.
\end{align*}
After a simple change of variables $(t \mapsto 1-t)$ and recalling the definition of constant $\varepsilon_{rq/2}$ in \eqref{def-eps}, we see that
\begin{equation}\label{Is}	
I_{2}  \leq   \varepsilon_{rq/2} \, A_2 \, \frac{rq}{2} \left(\frac{rq}{2}-1\right) \,  \langle v \rangle^2 \langle v_* \rangle^2 \,   \Big( \langle v \rangle^2 + \langle v_* \rangle^2 \Big)^{\frac{rq}{2}-2}.
\end{equation}
\noindent Putting together the estimate for $I_{1}$ and for $I_{2}$, we obtain the desired  estimate on the weight $G_{rq}(v,v_*)$.

\section{Ordinary differential inequalities for moments} \label{sec-ode} 

\noindent In this section we present two differential inequalities for polynomial moments (Proposition \ref{prop-ode}) which will 
be essential for the proof of Theorem \ref{thm}. We also state and prove a result about generation of polynomial moments in the non-cutoff case (Proposition \ref{prop-genpoly}). Before we state the proposition, we recall the ``floor function" of a real number, which in the case of a positive real number $x\in\R^+$ coincides with the integer part of $x$
\begin{align}
\lfloor x \rfloor := \mbox{integer part of} \; x.
\end{align}

\begin{proposition}\label{prop-ode}
Suppose all the assumptions of Theorem \ref{thm} are satisfied. Let $q \in \N$, and define  $k_p = \lfloor \frac{p +1}{2} \rfloor$ for any $p\in \R$ to be the integer part of $(p+1)/2$. Then for some constants $K_1, K_2, K_3 >0$ (depending on $\gamma$, $b(\cos \theta)$, dimension $d$) we have the following two ordinary differential inequalities for polynomial moments of the solution $f$ to the Boltzmann equation
\begin{itemize}
\vspace{-0.1in}

\item[{\bf (a)}] The ``$m_{\gamma k}$ version" needed for the generation of exponential moments
\begin{align}  \label{thm-b}
       	m'_{\gamma q}(t)   & \leq -K_1 \, m_{\gamma q+\gamma} \, + \,  K_2 \, m_{\gamma q}\\   
                    &+ K_3 \, \varepsilon_{q\gamma/2} \, \frac{q\gamma}{2} \left(\frac{q\gamma}{2}-1\right)
			\displaystyle \sum_{k=1}^{ 1+ k_{\frac{q}{2} - \frac{2}{\gamma}} } 
	      		\binom{\frac{q}{2} - \frac{2}{\gamma}}{k-1} 
	      		\left(m_{2\gamma k+\gamma} \, m_{\gamma q - 2\gamma k} 
	      		+ m_{2\gamma k} \, m_{\gamma q - 2\gamma k +\gamma}\right).\nonumber 
\end{align}
  
  \item[{\bf (b)}] The ``$m_{2k}$ version" needed for propagation of Mittag-Leffler moments
\begin{align} \label{thm-c}
m'_{2q} \leq -K_1 \, m_{2q+\gamma} \, &+ \,  K_2 \, m_{2q} \\
	  & + \, K_3 \, \varepsilon_q \, q(q-1)
	      \displaystyle \sum_{k=1}^{k_q} \binom{q-2}{k-1} 
	      \left(m_{2k+\gamma} \, m_{2(q-k)} + m_{2k} \, m_{2(q-k) +\gamma}\right).  \nonumber 
\end{align}
 \end{itemize}
 In both cases, the constant $K_1=A_2 C_\gamma$, where  $A_2$ was defined  in \eqref{part1} and $C_\gamma$, to be defined in the proof below, only depends on the $\gamma$ rate of the hard potentials. Similarly $K_2$ and $K_3$, also depend on  data, through the dependence on $A_2$ and $C_\gamma$.
\end{proposition}

\begin{prf} 
We start the proof by analyzing $m_{rq}$ with a general polynomial weight $\langle v \rangle^{rq}$. Then by setting $r=\gamma$ we shall derive (a) and by setting $r=2$  we shall obtain (b). 
Recall that  after multiplying the Boltzmann equation \eqref{eq-be} by $\langle v \rangle^{rq}$, the weak formulation \eqref{weak-Q}  yields
\begin{equation}\label{sym}
     m'_{rq}(t)  = \frac{1}{2}\displaystyle\iint\limits_{\R^{2d}} f f_* \; \, G_{rq}(v,v_*) \;dv\; dv_*.
 \end{equation}
The weight function $G_{rq}$ can be estimated as in Proposition \ref{averaged-pov}, which yields
\begin{align}\label{aftergrq} \nonumber
& m'_{rq}(t) \leq    - \frac{A_2}{2} \displaystyle\int_{\R^d} \int_{\R^d} f\; f_* \; |v-v_*|^\gamma \;
	  \Big( \langle v \rangle^{rq} +  \langle v_* \rangle^{rq} \Big) dv dv_*  \\  \nonumber
& \qquad+ \frac{A_2}{2} \displaystyle\int_{\R^d} \int_{\R^d} f\; f_* \; |v-v_*|^\gamma \;
	         \Big( \langle v \rangle^{rq-2} \langle v_* \rangle^{2} 
		    + \langle v \rangle^{2} \langle v_* \rangle^{rq-2} \Big) dv dv_* \\
& \qquad+ \,\frac{A_2}{2} \, \varepsilon_{rq/2} \, \frac{rq}{2} \left(\frac{rq}{2}-1\right)
	    \displaystyle\int_{\R^d} \int_{\R^d} f\, f_* \, |v-v_*|^\gamma \,		   
		 \langle v \rangle^2 \langle v_* \rangle^2 \,
	    \Big( \langle v \rangle^2 + \langle v_* \rangle^2 \Big)^{\frac{rq}{2}-2} 	dv dv_* 
\end{align}
We estimate $|v-v_*|^\gamma$ via elementary inequalities
\begin{equation}\label{eq-triangle}
       |v-v_*|^{\gamma} \leq C_\gamma^{-1} \big( \langle v \rangle^{\gamma} + \langle v_* \rangle^{\gamma} \big) 
       \qquad \mbox{and} \qquad
      |v-v_*|^{\gamma} \geq  C_\gamma \langle v \rangle^{\gamma} - \langle v_* \rangle^{\gamma}, 
 \end{equation}
 where $C_\gamma = \min\{1, 2^{1-\gamma}\}$  (see for example \cite{alcagamo13}). As an immediate consequence
 \begin{align}\label{withg1} \nonumber
 	& |v-v_*|^{\gamma}  \Big( \langle v \rangle^{rq} + \langle v_* \rangle^{rq} \Big)  \\ \nonumber
	& \qquad \qquad  \geq \Big( C_\gamma \langle v \rangle^{\gamma} - \langle v_* \rangle^{\gamma}  \Big)
		 \langle v \rangle^{rq} +   \Big( C_\gamma \langle v_* \rangle^{\gamma} - \langle v \rangle^{\gamma}  \Big)
		 \langle v_* \rangle^{rq} \vspace{5pt} \\
           & \qquad \qquad =  C_\gamma \Big( \langle v \rangle^{rq+\gamma} + \langle v_* \rangle^{rq + \gamma} \Big) 
               - \Big( \langle v \rangle^{rq} \langle v_* \rangle^{\gamma}  +  
                       \langle v \rangle^{\gamma} \langle v_* \rangle^{rq}\Big),
\end{align}
and
\begin{align}\label{withg2} \nonumber
	& |v-v_*|^{\gamma} \Big( \langle v \rangle^{rq-2} \langle v_* \rangle^{2} +
            	\langle v \rangle^{2} \langle v_* \rangle^{rq-2}\Big) \\ \nonumber
           & \qquad \qquad \leq C_\gamma^{-1} \Big( \langle v \rangle^{\gamma} + \langle v_* \rangle^{\gamma} \Big) \;
            	\Big( \langle v \rangle^{rq-2} \langle v_* \rangle^{2} +  \langle v \rangle^{2} \langle v_* \rangle^{rq-2}\Big) 	
		\vspace{5pt} \\          
           & \qquad \qquad \leq 2 C_\gamma^{-1} \Big( \langle v \rangle^{rq} \langle v_* \rangle^{\gamma} +
                  \langle v \rangle^{\gamma} \langle v_* \rangle^{rq}\Big),
\end{align}
where the last inequality uses Lemma \ref{trick}.
Combining \eqref{aftergrq} with \eqref{withg1} and \eqref{withg2} we obtain
\begin{align*}
& m'_{rq}(t) \leq   - \; \frac{A_2}{2} \; C_\gamma 
	 \displaystyle\int_{\R^d} \int_{\R^d} ff_* 
	    \Big( \langle v \rangle^{rq+\gamma} + \langle v_* \rangle^{rq + \gamma} \Big)  dv dv_*     \\ \nonumber
 & \qquad+ \,  \frac{A_2}{2}  (1+ 2C_\gamma^{-1}) \displaystyle\int_{\R^d} \int_{\R^d} f\; f_* \; 
		    \Big( \langle v \rangle^{rq} \langle v_* \rangle^{\gamma} +
                  \langle v \rangle^{\gamma} \langle v_* \rangle^{rq}\Big) dv dv_* \\  \nonumber
&\qquad + \; \frac{A_2 \;  \varepsilon_{rq/2}}{ 2 \, C_\gamma} \,  \frac{rq}{2} \left(\frac{rq}{2}-1\right)
		\displaystyle\iint_{\R^{2d}}  f f_* 
		\Big(\langle v \rangle^\gamma + \langle v_* \rangle^\gamma \Big) \;   
                  \langle v \rangle^2 \langle v_* \rangle^2 \;
	    \Big( \langle v \rangle^2 + \langle v_* \rangle^2 \Big)^{\frac{rq}{2}-2} dv dv_*  \\  \nonumber
\ 
&\qquad \leq   - \; \frac{A_2}{2} \; C_\gamma m_{0}(t)  m_{rq+\gamma}(t)  +    \frac{A_2}{2}  (1+ 2C_\gamma^{-1}) m_\gamma(t) m_{rq}(t)  \\
&\quad\qquad 
+\frac{A_2 \;  \varepsilon_{rq/2}}{ 2 \, C_\gamma} \,  \frac{rq}{2} \left(\frac{rq}{2}-1\right)
		\displaystyle\iint_{\R^{2d}}  f f_* 
		(\langle v \rangle^\gamma + \langle v_* \rangle^\gamma) 
                  \langle v \rangle^2 \langle v_* \rangle^2 
	    ( \langle v \rangle^2 + \langle v_* \rangle^2 )^{\frac{rq}{2}-2} dv dv_*. 	 
\end{align*}
Therefore, since  $0<\gamma\leq 1$,  by conservation of mass and energy, $m_0(t)=m_0(0)$ and $m_{\gamma}(t)\leq m_2(0)$, 
\begin{align}	 \label{thm-a}
& m'_{rq}(t) \; \leq \;  -K_1 \, m_{rq+\gamma}(t) + K_2 \, m_{rq}(t)
	+ \,\frac{ K_3}{2}\, \varepsilon_{rq/2} \,  \frac{rq}{2} \left(\frac{rq}{2}-1\right) \\ \nonumber
& \qquad \qquad  \,
		\displaystyle\iint_{\R^{2d}}  ff_* 
		(\langle v \rangle^\gamma + \langle v_* \rangle^\gamma )   
                  \langle v \rangle^2 \langle v_* \rangle^2 
	    ( \langle v \rangle^2 + \langle v_* \rangle^2 )^{\frac{rq}{2}-2}dv dv_*,	  
\end{align}
where $K_1 = A_2 \, C_\gamma \, m_0(0)$, \; $K_2 = A_2 \, (1 + 2C^{-1}_\gamma) \, m_2(0)$, and \; $K_3 = \frac{A_2}{ C_\gamma}$, so 
 these three constants only depend on the initial mass and energy, on the rate of the potential $\gamma$ and on the angular singularity condition \eqref{eq-nc} 
 that determines the constant $A_2$.

From here,  we proceed to prove (a) and (b) separately.

\noindent \textbf{(a)} Setting $r=\gamma$ in \eqref{thm-a}, applying the following elementary polynomial inequality which is valid for $\gamma\in(0,1]$
 \begin{equation}\label{2tos}
   \Big( \langle v \rangle^2 + \langle v_* \rangle^2 \Big)^{\frac{\gamma q}{2}-2} \, \leq \,
     \Big( \langle v \rangle^{2\gamma} + \langle v_* \rangle^{2\gamma} \Big)^{\frac{q}{2} - \frac{2}{\gamma}},
 \end{equation}
and using  the polynomial Lemma \ref{lm-comb} yields
\begin{align*}
 & m'_{\gamma q}(t)   \leq  \; -\; K_1 \; m_{\gamma q+\gamma} + K_2 \;  m_{\gamma q}  
	 +  \frac{K_3}{2} \; \varepsilon_{\gamma q/2}  \; \frac{\gamma q}{2} \left(\frac{\gamma q}{2}-1\right)   \\
& \qquad\qquad\qquad \displaystyle\iint\limits_{\R^{2d}} f\; f_* 
	\Big(\langle v \rangle^\gamma + \langle v_* \rangle^\gamma \Big) \;
	\langle v \rangle^2 \langle v_* \rangle^2 \;
	\Big( \langle v \rangle^{2\gamma} \, + \, \langle v_* \rangle^{2\gamma}\Big)^{\frac{q}{2} - \frac{2}{\gamma}} \; dv dv_*\\          
&\leq  -\; K_1 \; m_{\gamma q+\gamma} + K_2 \; m_{\gamma q}      
	      + \,\frac{K_3}{2} \, \varepsilon_{\gamma q/2} \;  \frac{\gamma q}{2} \left(\frac{\gamma q}{2}-1\right)
		 \displaystyle\iint\limits_{\R^{2d}} f\; f_* \; 
		\Big(\langle v \rangle^\gamma + \langle v_* \rangle^\gamma \Big)  \\
& \qquad \qquad
		\displaystyle \sum_{k=0}^{k_{\frac{q}{2} - \frac{2}{\gamma}}} 
		\binom{\frac{q}{2} - \frac{2}{\gamma}}{k}
		\Big( \langle v \rangle^{2\gamma k+2} \langle v_* \rangle^{\gamma q-2\gamma k -2}  +
		 \langle v \rangle^{\gamma q-2\gamma k -2} \langle v_* \rangle^{2\gamma k+2}\Big) \; dv dv_* \\
& \leq  -\; K_1 \; m_{\gamma q+\gamma} + K_2 \; m_{\gamma q}      
	      +   \, K_3 \, \varepsilon_{\gamma q/2} \;  \frac{\gamma q}{2} \left(\frac{\gamma q}{2}-1\right) \cdot \\
&  \qquad \qquad
		 \displaystyle \sum_{k=0}^{k_{\frac{q}{2} - \frac{2}{\gamma}}} 
		\binom{\frac{q}{2} - \frac{2}{\gamma}}{k}
		\Big( m_{2\gamma k + 2 + \gamma} \, m_{\gamma q - 2 \gamma k - 2} + 
			m_{\gamma q - 2\gamma k - 2 + \gamma} \, m_{2\gamma k +2}\Big) \; dv dv_*.
\end{align*}
Finally, re-indexing $k$ to $k-1$ and applying Lemma \ref{trick} yields
\begin{align*}
 m'_{\gamma q}(t)   \;  & \leq  -\; K_1 \; m_{\gamma q+\gamma} + K_2 \; m_{\gamma q}      
	      +   \, K_3 \, \varepsilon_{\gamma q/2} \;  \frac{\gamma q}{2} \left(\frac{\gamma q}{2}-1\right)  \\
&  \qquad \qquad
		\displaystyle \sum_{k=1}^{1+k_{\frac{q}{2} - \frac{2}{\gamma}}} 
		\binom{\frac{q}{2} - \frac{2}{\gamma}}{k-1}
		\Big( m_{2\gamma k + \gamma} \, m_{\gamma q - 2 \gamma k } + 
			m_{\gamma q - 2\gamma k + \gamma} \, m_{2\gamma k}\Big) \; dv dv_*.
\end{align*}
which completes proof of (a).

 \noindent {\bf (b)} Now, we set $r=2$  in \eqref{thm-a} and apply Lemma \ref{lm-comb} to obtain
\begin{align*}
&m'_{2q}(t)   \leq  \; -\; K_1 \; m_{2q+\gamma} + K_2 \; m_{2q} 
	  +  \; K_3 \; \varepsilon_q \; q(q-1) 
	     \displaystyle\iint\limits_{\R^{2d}} f\; f_* 
		 \Big(\langle v \rangle^\gamma + \langle v_* \rangle^\gamma \Big)  \\
& \qquad \qquad \ \langle v \rangle^2 \langle v_* \rangle^2 \;
	        \displaystyle \sum_{k=0}^{k_{q-2}} \binom{q-2}{k}
	        \left( \langle v \rangle^{2k} \; \langle v_* \rangle^{2(q-2)-2k} +
	        \langle v \rangle^{2(q-2)-2k} \langle v_* \rangle^{2k} \right) \; dv dv_* \\ 
& =-\; K_1 \; m_{2q+\gamma} + K_2 \; m_{2q} + \; K_3 \; \varepsilon_q \; q(q-1) \   		\displaystyle\iint\limits_{\R^{2d}} f\; f_* \; 
		\left(\langle v \rangle^\gamma + \langle v_* \rangle^\gamma \right)\  \\
& \qquad \qquad
		\  \displaystyle \sum_{k=0}^{k_{q-2}} \binom{q-2}{k}
		\Big( \langle v \rangle^{2k+2} \langle v_* \rangle^{2q-2k -2}  +
		 \langle v \rangle^{2q-2k -2} \langle v_* \rangle^{2k+2}\Big) \; dv dv_* \\ 
&=  -\; K_1 \; m_{2q+\gamma} + K_2 \; m_{2q} + \; K_3  \; \varepsilon_q \; q(q-1) \,
	\displaystyle \sum_{k=1}^{k_q} \binom{q-2}{k-1}
	\big( m_{2k+\gamma}\, m_{2q-2k} + m_{2k}\, m_{2q-2k+\gamma}\big).
\end{align*}
The last equality is obtained by re-indexing $k$ to $k-1$ and using that $1+ k_{q-2} = k_q$. This completes proof of (b).
\end{prf}

\medskip

\begin{proposition}[Polynomial moment bounds for the non-cutoff case] \label{prop-genpoly}
Suppose all the assumptions of Theorem \ref{thm} are satisfied. Let $f$ be solution to the homogeneous Boltzmann equation \eqref{eq-be} associated to the  initial data $f_0$.
\vspace{-0.1in}
\begin{enumerate}
	\item Let the initial mass and energy be finite, i.e. $m_2(0)$ bounded, then   for  every  $p>0$  there exists a constant $\bB_{rp} \geq 0$, depending on $2^{rp}$, $\gamma$, $m_2(0)$ and $A_2$ from condition   \eqref{eq-nc}, such that
		\begin{equation} \label{poly-gen}
			m_{rp}(t) \, \leq \, {\bB}_{rp} \, \max \{ 1, t^{-rp/\gamma}\}, \quad  \mbox{for all}\ \  r \in {\mathbb R^+}  \ \ \mbox{and}  \   t \ge 0 \, .
		\end{equation}
	\item Furthermore, if $m_{rp}(0)$ is finite, then the control can be improved to
		\begin{equation}\label{poly-prop}
			m_{rp}(t) \, \leq \, {\bB}_{rp}, \quad   \mbox{for all} \ \ r\in {\mathbb R^+}  \ \ \mbox{and}  \   t \ge 0 .
		\end{equation}
\end{enumerate}
 \end{proposition}
 
\begin{prf} These statements can be shown by studying comparison theorems for initial value problems associated with ordinary differential inequalities of the type 
$$
y'(t) + A y^{1+c}(t) \le B y(t),
$$
and comparing them to classical  Bernoulli's differential equations for the same given initial $y(0)$.  In our context, these inequalities are a result of estimating moments for variable hard potentials, i.e. $\gamma>0$ as indicated in \eqref{eq-gamma}.
Comparison with Bernoulli type differential equations was classically used  in angular cutoff cases in \cite{we96, we97, miwe99, alcagamo13}. Also it was used in the proof of propagation of $L^1$ exponential tails for the derivatives of the solution of the Boltzmann equation  by means of  geometric series methods in \cite{bogapa04,gapavi09,alga08}.  
 
 In fact, the extension to the non-cutoff case follows  in a straightforward way from the moments estimates in Proposition~\ref{prop-ode}. 
   Indeed,  the moment estimates, from either \eqref{thm-b} or  \eqref{thm-c}, show  that the only  negative contribution is on the highest order moment,   being  either $m_{r q+\gamma}$ with $\gamma>0$ for $r=\gamma$ or
$2$, respectively.  Then, due to the fact that  $\gamma>0$, an application of classical Jensen's inequality with the convex function $\varphi(x)=x^{1+\gamma/(rp)}$  yields 
  \begin{align*}
  m_{rp+\gamma}(t)  \ \geq \ m^{-\gamma/(rp)}_0(0)\, m_{rp}^{1 + \gamma/(rp)}(t) \ \  \mbox{for all} \ t > 0 \, .
  \end{align*}
Applying this estimate to the negative term in either  \eqref{thm-b} or  \eqref{thm-c},  results in the following estimate
   \begin{align}\label{mom-bd}
m'_{rp} & \;  \leq \; B_{rp} m_{rp} - K_1 m_{rp+\gamma}\; \leq \; B_{rp} m_{rp} - K_1 m_{rp}^{1 + \gamma/(rp)},
\end{align}
with  $r$ either $\gamma$  in  \eqref{thm-b}, or $2$ in  \eqref{thm-c}.  The constants  are $K_1=K_1(\gamma, A_2)$  with $0<\gamma\le 1$,  and  
$A_2$ from the angular integrability condition \eqref{eq-nc}; and   $B_{rp} =B_{rp}(K_2, 2^{rp} K_3)$, after using that  
$\varepsilon_p \le 1,$  where $K_2$ and $K_3$ also depend on the initial data and collision kernel through $\gamma$ and $A_2$.
 
Therefore, as in  \cite{we96}, we set  
$$
y(t):= m_{rp}(t),\ A:= K_1, \ B:=B_{rp} \mbox{ and } c=\gamma/(rp).
$$
 The bound \eqref{poly-prop} then follows by finding an upper solution
 that solves the associated  Bernoulli ODE 
  $$
 y'(t) = B y(t) -  A y^{1+c}(t)
 $$ 
 with finite initial polynomial moment  $y(0)=m_{rp}(0)$. This yields that for any $t>0$
\begin{align}
{m}_{rp}(t) 
&\leq \left[m_{rp}^{-\gamma/(rp)}(0) \, e^{-t \, B\gamma/(rp)}  + \frac AB (1-  e^{-t \, B\gamma/(rp)}) \right]^{-rp/\gamma} 	\nonumber \\
& \leq  \left[\frac AB (1-  e^{-t \, B\gamma/(rp)}) \right]^{-rp/\gamma} 	\nonumber \\
& \leq \left( \frac{A}{B}\right)^{-rp/\gamma} 
	\left\{
		\begin{array}{l}
			 \left( \frac{rp}{B\gamma} e^{B\gamma/rp}\right)^{-rp/\gamma} \,\, t^{-rp/\gamma} , \quad t < 1, \vspace{5pt}\\ 
			 (1-  e^{- \, B\gamma/(rp)})^{-rp/\gamma} , \quad t\geq 1. 
		\end{array}
	\right. \nonumber \\
&\le \bB_{rp} \max\{1,  t^{-rp/\gamma} \}, \label{bBrp}
\end{align}
where $ \bB_{rp} := \left( \frac{K_1}{B_{rp}}\right)^{-rp/\gamma}  \max \left\{  \left( \frac{rp}{\gamma B_{rp}} e^{\gamma B_{rp}/rp}\right)^{-rp/\gamma},  \,(1-  e^{- \, \gamma B_{rp}/(rp)})^{-rp/\gamma}\right\}$.

Now, since $m_{rp}(t)$ is a continuous function of time, if $m_{rp}(0)$ is finite for any  $rp\geq 1$, then the bound for strictly positive times we just obtained in \eqref{bBrp} implies
 \begin{align}\label{bBrp2}
 {m}_{rp}(t) \leq  \bB_{rp} .
 \end{align}
for possibly different constants $\bB_{rp}$. We finally stress that constants ${\bf B}_{rp}$  depend on $2^{rp}, \gamma, m_{2}(0)$ and $A_2$ from condition   \eqref{eq-nc}.
\end{prf}

\vspace{20pt}

\section{Proof of  Mittag-Leffler moments' propagation} \label{sec-ml}


\begin{proof}[Proof of Theorem \ref{thm} (b)]
\noindent Let us recall representation \eqref{def-mlmomsum} of the Mittag-Leffler moment of order $s$ and rate $\alpha$ in terms of infinite sums  
\begin{align}
\displaystyle\int_{\R^d}  f(t,v) \;\; \cE_{2/s} (\alpha^{2/s} \, \bra v \ket^2 ) \; dv
		\;\; = \;\; \displaystyle \sum_{q=0}^\infty \frac{m_{2q}(t) \, \alpha^{2q/s}}{\Gamma(\frac{2}{s} \, q+1)}.	
\end{align}
We introduce abbreviated notation $a = \frac{2}{s}$, and note that since $s\in (0, 2)$, we have
\begin{align}
1 < a:= \frac{2}{s} < \infty.
\end{align}
We consider the $n$-th partial sum, denoted by $\cE^n_a$, and the corresponding sum, denoted by $\cI^n_{a,\gamma}$, in which polynomial moments are shifted by  $\gamma$. In other words, we consider
\begin{align*}
 	\cE^n_a (\alpha,t) = \displaystyle \sum_{q=0}^n \frac{m_{2 q}(t) \; \alpha^{aq}}{\Gamma(aq+1)}, \qquad
 	\cI^n_{a,\gamma} (\alpha,t) = \displaystyle \sum_{q=0}^n \frac{m_{2 q + \gamma}(t) \; \alpha^{aq}}{\Gamma(aq+1)}. 
\end{align*}
For each $n \in \N$, define
\begin{equation}\label{def-Tn}
 	 	T_n:=  \sup \left\{ t \geq 0 \;|\; \cE^n_a(\alpha, \tau) <  4M_0, \; \mbox{for all} \; \tau \in [0,t) \right\}.
\end{equation}
where the constant $M_0$ is the one from the initial condition \eqref{eq-id}.

This  parameter $T_n$ is well-defined and positive. Indeed,  since $\alpha$ will be chosen to be, at least, smaller than $\alpha_0$, then at time $t=0$ we have 
  \begin{equation*} 
\cE^n_a(0) =  \sum_{q=0}^n \frac{m_{2q}(0) \; \alpha^{aq}}{\Gamma(aq+1)} \ <\  \sum_{q=0}^\infty \frac{m_{2q}(0) \; \alpha_0^{aq}}{\Gamma(aq+1)}= \int f_0(v)  \cE_{2/s} (\alpha_0^{2/s} \, \bra v \ket^2 )  \, dv < 4\, M_0, 
\end{equation*}
uniformly in $n$. Therefore, since partial sums are continuous functions of time (they are finite sums and each $m_{2q}(t)$ is also continuous function in time $t$), we conclude that $\cE^n_a(\alpha, t) <4M_0$ holds for $t$ on some positive time interval denoted $[0, t_n)$ with $t_n>0$  (and hence $T_n>0$).

Next, we look for an ordinary differential inequality that the partial sum $\cE^n_a(\alpha, t)$ satisfies, following the steps presented in Subsection \ref{sub-strategy}. We start by splitting $\frac{d}{dt}\cE^n_a(\alpha, t)$ into the following two sums, where index $q_0$ will be fixed later, and then apply the moment differential inequality \eqref{thm-c}
\begin{align} 
&\displaystyle \frac{d }{dt} \cE^n_a(\alpha, t) \,  = \,
   	\displaystyle\sum_{q=0}^{q_0-1} \frac{m'_{2 q}(t) \, \alpha^{aq}}{\Gamma(aq+1)} 
        	\;+\; 
	\displaystyle\sum_{q=q_0}^{n} \frac{m'_{2 q}(t) \  \alpha^{aq}}{\Gamma(aq+1)} \nonumber \\ 
& \leq \,
	\displaystyle\sum_{q=0}^{q_0-1} \frac{m'_{2 q}(t) \, \alpha^{aq}}{\Gamma(aq+1)}
	\, - \, K_1 \displaystyle\sum_{q=q_0}^n \frac{m_{2 q + \gamma}(t) \, \alpha^{aq}}{\Gamma(aq+1)} 
	+ \, K_2 \displaystyle\sum_{q=q_0}^n \frac{m_{2 q}(t) \, \alpha^{aq}}{\Gamma(aq+1)}  \nonumber \\
& \qquad + \, K_3 \displaystyle \sum_{q=q_0}^{n} 
	    \frac{\varepsilon_q \, q (q-1) \, \alpha^{aq}}{\Gamma(aq+1)} 	    
	    \sum_{k=1}^{k_q} \binom{q - 2}{k-1} 
	 \Big(m_{2k+\gamma}\, m_{2(q-k)} + m_{2k} m_{2(q-k)+\gamma} \Big)  \nonumber \\
& =: \, S_0 \, - \, K_1 \, S_1 \, + \, K_2 \, S_2 \, + \, K_3 \, S_3. \label{prop-4sums}
\end{align}
We estimate each of the four sums $S_0, S_1, S_2$ and $ S_3$ separately, with the goal of comparing each of them to the functions $ \cE^n_a(\alpha, t)$ and $ \cI^n_{a,\gamma}(\alpha, t)$. We remark that the most involving term is $S_3$. It resembles the corresponding sum in the anglar cutoff case  \cite{alcagamo13}, with a crucial difference that our sum $S_3$ has two extra powers of $q$, namely $q(q-1)$. Therefore, a very sharp calculations is required to control   the growth of $S_3$ as a function of the number $q$ of moments. This is achieved    by an appropriate renormalization of polynomial moments within $S_3$ and also by invoking the decay rate of associated combinatoric sums of Beta functions developed in the Appendix \ref{appx}.

The term $S_0$ can be bounded by terms that depends on the initial data and the parameters of the collision cross section.
Indeed,  from Lemma \ref{prop-genpoly}, the propagated polynomial moments can be estimated  as follows
\begin{align}\label{prop-s00}
m_p \leq \bB_{p} \ \ \mbox{ and }  \ \  \ m'_{p} \leq  B_{p} \, \bB_{p},\qquad \mbox{for any} \, p>0,
\end{align}
where  the constant ${\bf B}_{p}$ defined in \eqref{bBrp} depends on $\gamma$, the initial $p$-polynomial moment $m_{p}(0)$ and $A_2$ from condition   \eqref{eq-nc}.  

In particular,  for $0<\gamma<1$, we can fix  $q_0$, to be chosen later,  such that the constant 
\begin{align}
 c_{q_0}: = \max_{p \in I_{q_0}} \{\bB_{p}, \,   B_{p} \, \bB_{p}\},  \ \  \ \mbox{with } \ \ I_{q_0} = \{ 0, \dots,  2q_0 +1 \}
\end{align}
 depends only on $q_0$, $\gamma$, $A_2$ from condition  \eqref{eq-nc}, and  the initial polynomial moments $m_{q}(0)$, for  $q\in I_{q_0}$. 
 Thus,  due to the monotoncity of $L^1_k$ norms with respect to $k$ as presented in \eqref{moml1}, both the $2q$-moments and its derivatives, as well as the shifted moments of order $2q+\gamma$,  are  controlled by $c_{q_0}$  as follows
\begin{align}\label{cq0_1}
& m_{2q}(t) , \, m_{2q+\gamma}(t), \,    m'_{2q}(t)   \leq c_{q_0}, \quad \mbox{for all } \, q \in \{0,1,2, ... q_0\}, 
\end{align}

Therefore, for $q_0$ fixed, to be chosen later, $S_0$ is estimated by
\begin{equation}\label{prop-s1}
\begin{array}{l}
	S_0 \, := \,\displaystyle \sum_{q=0}^{q_0 -1} \frac{m'_{2q} \; \alpha^{aq}}{\Gamma(aq+1)} 
		\, \leq \, c_{q_0} \, \sum_{q=0}^{q_0 -1} \frac{\alpha^{aq}}{\Gamma(aq+1)} \\

	\qquad \leq \,  c_{q_0} \,\displaystyle \sum_{q=0}^{q_0 -1} \frac{(\alpha^a)^q}{\Gamma(q+1)}
		\, \leq \, c_{q_0} \, e^{\alpha^a} \, \leq 2 \, c_{q_0} \, ,
\end{array}
\end{equation}
for the parameter  $\alpha$ small enough to satisfy
\begin{equation}\label{alpha2}
   \alpha <   (\ln{2})^{1/a}  \ , \ \ \mbox{ or equivalently,}  \ \  e^{\alpha^a} \leq 2.
\end{equation}

The second term $S_1$ is crucial, as it brings the negative contribution that will yield uniform in $n$ and  global in time control to an ordinary differential inequality for  $ \cE^n_a(\alpha, t)$.
In fact, $S_1$  is controlled from below by   $ \cI^n_{a\gamma}(\alpha, t)$ as follows.
\begin{equation*}
	S_1 \, := \,\displaystyle \sum_{q=q_0}^n \frac{m_{2q+\gamma} \; \alpha^{aq}}{\Gamma(aq+1)} 
		\, = \, \cI^n_{a,\gamma} \, - \,
			\displaystyle \sum_{q=0}^{q_0-1} \frac{m_{2q+\gamma} \; \alpha^{aq}}{\Gamma(aq+1)}, 
\end{equation*}
so using \eqref{cq0_1} and the estimate just obtained for $S_0$ in \eqref{prop-s1}, yields the bound from below  
\begin{equation}\label{prop-s2}
	S_1 \, \geq  \, \cI^n_{a,\gamma} \, - \,
		c_{q_0} \sum_{q=0}^{q_0-1} \frac{\alpha^{aq}}{\Gamma(aq+1)} 
		\, \geq \,  \cI^n_{a,\gamma} \, - \, 2c_{q_0}.
\end{equation}

The sum $S_2$ is a part of the partial sum $\cE^n_a$, so
\begin{equation}\label{prop-s3}
	S_2 \, \leq \, \cE^n_a.
\end{equation}
While this term is positive it will need to be lower order than the one in the negative part of the right hand side.

Finally, we estimate $S_3$ and show that it can be bounded by the product of $ \cE^n_a(\alpha, t)$ and $ \cI^n_{a,\gamma}(\alpha, t)$.  We work out the details of the first term in the sum $S_3 := S_{3,1} + S_{3,2} $, that is the one with $m_{2k+\gamma} \, m_{2(q-k)}$. The other sum with $m_{2k} \, m_{2(q-k)+\gamma}$ can be bounded by following a similar strategy. 
In order to generate both the partial sum $ \cE^n_a(\alpha, t)$ and the shifted one
$ \cI^n_{a,\gamma}(\alpha, t)$, we make use of the following well known  relations between Gamma and Beta functions (see also Appendix \ref{appx}).
  \begin{align}\label{beta-S4}
  B(ak+1, a(q-k)+1)\  &=\  \frac{\Gamma(ak+1) \, \Gamma(a(q-k)+1)}{\Gamma(\,(ak +1) + (a(q-k) +1) \, )} \nonumber \\
  &\ \ \\
  &= \  \frac{\Gamma(ak+1) \, \Gamma(a(q-k) +1)}{\Gamma(aq+2)}\  . \nonumber
  \end{align}
Therefore, multiplying  and dividing 
products of moments 
$m_{2 k+\gamma}  m_{2(q-k)}$  in $S_{3,1}$,
by  ${\Gamma(ak+1) \Gamma(a(q-k)+1)}$ 
yields

\begin{align*}
	S_{3,1}  & := \, \displaystyle \sum_{q=q_0}^n 	
		\frac{\varepsilon_q \, q \, (q-1) \, \alpha^{aq}}{\Gamma(aq+1)}
		\, \sum_{k=1}^{k_q} \binom{q-2}{k-1} m_{2k + \gamma} \, m_{2(q-k)}  \\
	& \, = \,  \displaystyle\sum_{q=q_0}^n \varepsilon_q \, q \, (q-1)
		\sum_{k=1}^{k_q} \binom{q-2}{k-1}
			\frac{m_{2k + \gamma} \alpha^{ak}}{\Gamma(ak+1)} \,
			\frac{m_{2(q-k)} \alpha^{a(q-k)}}{\Gamma(a(q-k)+1)}\, \\
		&\qquad\qquad\qquad\qquad\qquad \,B(ak+1, a(q-k)+1) \,
			\displaystyle\frac{\Gamma(aq+2)}{\Gamma(aq+1)}\, .
\end{align*}
Note that the factors $\frac{m_{2k + \gamma} \alpha^{ak}}{\Gamma(ak+1)}$ and $\frac{m_{2(q-k)} \alpha^{a(q-k)}}{\Gamma(a(q-k)+1)}$ are the building blocks of 
$ \cI^n_{a,\gamma}(\alpha, t)$ and $ \cE^n_a(\alpha, t)$, respectively.

Next, since $\Gamma(aq+2)/\Gamma(aq+1) = aq+1$, using the inequality $\sum_k a_k \, b_k \leq \sum_k a_k \, \sum_k b_k,$ it follows that
\begin{equation}\label{S4-beforeBe}
\begin{array}{l}
	S_{3,1} \, \leq \,  \displaystyle \sum_{q=q_0}^n 	\varepsilon_q \, (aq+1) \, q\, (q-1)
		\left( \displaystyle\sum_{k=1}^{k_q} \binom{q-2}{k-1} \, B(ak+1, a(q-k)+1)  \right) \,  \\
		\qquad \qquad\qquad\qquad\qquad \qquad \qquad 
		\left( \displaystyle \sum_{k=1}^{k_q} \frac{m_{2k + \gamma} \alpha^{ak}}{\Gamma(ak+1)} \,
			\frac{m_{2(q-k)} \alpha^{a(q-k)}}{\Gamma(a(q-k)+1)}  \right). 
\end{array}
\end{equation}

Next we  show that the factor 
$$ (aq+1) \, q\, (q-1)
		\left( \displaystyle\sum_{k=1}^{k_q} \binom{q-2}{k-1} \, B(ak+1, a(q-k)+1)  \right) \,
$$ 
on the right hand side of \eqref{S4-beforeBe} grows at most as $q^{2-a}$. 		 
Indeed,  using Lemma \ref{lem-sumB1},  the sum of the Beta functions  is bounded by $C_a (aq)^{-(1+a)}$.
Therefore,  $S_{3,1}$ is estimated by
\eqn\label{S41}
	S_{3,1} \, \le \, \displaystyle  C_a \sum_{q=q_0}^n 	\varepsilon_q \, q^{2-a} 
		\left( \displaystyle \sum_{k=1}^{k_q} \frac{m_{2k + \gamma} \alpha^{ak}}{\Gamma(ak+1)} \,
			\frac{m_{2(q-k)} \alpha^{a(q-k)}}{\Gamma(a(q-k)+1)}  \right),
\eeqn
where $C_a$ is a (possibly different) constant that depends on $a$. Now,  by Lemma  \ref{averaged-pov},  the factor $ \varepsilon_q \, q^{2-a}$ decreases monotonically to zero as $q\rightarrow \infty$ if the angular kernel $b(\cos{\theta})$ satisfies \eqref{eq-nc} with $\beta = 2a -2$. Hence,
\begin{equation}\label{q01}
\varepsilon_{q} \, q^{2-a}  \leq \varepsilon_{q_0} \, q_0^{2-a}    \, ,   \qquad\ \mbox{for any} \ q\ge q_0 \, ,
\end{equation}
and thus  the term $S_{3,1} $ is further estimated by
\begin{align*}
	S_{3,1} \leq  C_a \, \varepsilon_{q_0} \; q_0^{2-a}  \displaystyle \sum_{q=q_0}^n
		\displaystyle \sum_{k=1}^{k_q} \frac{m_{2k + \gamma} \alpha^{ak}}{\Gamma(ak+1)} \,
			\frac{m_{2(q-k)} \alpha^{a(q-k)}}{\Gamma(a(q-k)+1)}.
\end{align*}
Finally, inspired by \cite{alcagamo13}, we bound this double sum by the product of partial sums $\cE^n_a \, \cI^n_{a,\gamma}$. To achieve that, change the order of summation to obtain
\eqn\label{prop-s4}
\begin{array}{lcl}
	S_{3,1} &\le &  C_a\,\varepsilon_{q_0} \; q_0^{2-a} \, \displaystyle \sum_{k=0}^{k_n} 
		\sum_{\max\{q_0, 2k-1\}}^n
		 \frac{m_{2k + \gamma} \alpha^{ak}}{\Gamma(ak+1)} \,
		\frac{m_{2(q-k)} \alpha^{a(q-k)}}{\Gamma(a(q-k)+1)} \\

	&\leq&   C_a \, \varepsilon_{q_0} \; q_0^{2-a} \,  \displaystyle \sum_{k=0}^{k_n} 
		\frac{m_{2k + \gamma} \alpha^{ak}}{\Gamma(ak+1)} \,
		\sum_{\max\{q_0, 2k-1\}}^n
		\frac{m_{2(q-k)} \alpha^{a(q-k)}}{\Gamma(a(q-k)+1)} \\ \\
	&\leq&  C_a \,  \varepsilon_{q_0} \; q_0^{2-a} \, \cI^n_{a,\gamma} \,\cE^n_a \, ,
\end{array}
\eeqn
obtaining the expected control of $S_{3,1}$. As mentioned above the estimate of the companion sum  $S_{3,2}$ follows in a similar way, so we can assert
\begin{equation}\label{q01}
S_3 \, \le \, C_a\,  \varepsilon_{q_0} \; q_0^{2-a} \,\cE^n_a (t) \, \cI^n_{a,\gamma}(t). 
\end{equation}

Next we obtain an ordinary differential inequality for $\cE^n_a (t)$ depending only on data parameters and $\cI^n_{a,\gamma}(t)$. Indeed, combining \eqref{prop-s1}, \eqref{prop-s2}, \eqref{prop-s3} and \eqref{prop-s4} with \eqref{prop-4sums} yields 
\begin{equation}\label{prop-ei}
\displaystyle \frac{d }{dt} \cE^n_a  \, \leq \, 
     \,  - K_1 \,\cI^n_{a,\gamma} \, + \,  2 \,c_{q_0}(1+K_1) \, + \,  K_2 \,  \cE^n_a
	  \, + \,  \varepsilon_{q_0} \; q_0^{2-a} \,C_a \,  K_3  \,  \cI^n_{a,\gamma}  \,\cE^n_a \, .
\end{equation}
  
Since, by the definition of time $T_n$, the partial sum $\cE^n_a$ is bounded by the constant $4M_0$ on the time interval $[0, T_n]$, we can estimate, uniformly in $n$, the following two terms in \eqref{prop-ei}
 \begin{equation}\label{q03}
2 \,c_{q_0}(1+K_1) +  K_2 \,  \cE^n_a \le 2 \,c_{q_0}(1+K_1) + 4  K_2 \,M_0 =:  \cK_0,
  \end{equation} 
where $\cK_0$ depends only on the initial data and $q_0$ (still to be determined). 

Thus, factoring out $\cI^n_{a, \gamma}$ from the remaining two terms in 
\eqref{prop-ei} yields
\begin{align}\label{prop-ei2}
\displaystyle \frac{d }{dt} \cE^n_a 
	& \, \leq \, - \,\cI^n_{a, \gamma} \, \Big( K_1 \, - \, \varepsilon_{q_0} \; q_0^{2-a} \, C_a \, K_3 \, \cE^n_a \Big)
		 \, + \, \cK_0 \nonumber \\
	& \leq \,  - \,\cI^n_{a, \gamma} \, \Big( K_1 \, - \, 4 \varepsilon_{q_0} \; q_0^{2-a} \, C_a \, K_3 \, M_0 \Big)
		 \, + \, \cK_0,
\end{align}
where in the last inequality we again used that, by the definition of $T_n$, we have $\cE^n_a \leq 4M_0$ on the closed interval $[0, T_n]$.  Now, since $\varepsilon_{q_0} \; q_0^{2-a}$ converges to zero as $q_0$ tends to infinity (by Lemma  \ref{averaged-pov} as $b(\cos{\theta})$ satisfies \eqref{eq-nc} with $\beta = 2a -2$ ), we can choose large enough $q_0$ so that
 \begin{equation}\label{q04} 
K_1 \, - \, 4 \varepsilon_{q_0} \; q_0^{2-a} \, C_a \, K_3 \, M_0>   \frac{K_1}2.
   \end{equation} 
For such choice of $q_0$ we then have
\eqn \label{prop-minus i}
	\displaystyle \frac{d }{dt} \cE^n_a
	  \, \leq - \, \frac{K_1}{2} \, \cI^n_{a, \gamma}  \, + \, \cK_0 \; .
\eeqn

The final step consists  in finding  a lower bound for $\cI^n_{a,\gamma}$ in terms of $\cE^n_a$. 
The following calculation follows from a revised form of the lower bound given in \cite{alcagamo13},
\begin{equation}\label{prop-lower}
 \begin{array}{lcl}
   \cI^n_{a, \gamma}(t) &:=& \displaystyle \sum_{q=0}^n \frac{m_{2 q + \gamma} \  \alpha^{aq}}{\Gamma(aq+1)} 
    
      \; \geq \;  \displaystyle   \sum_{q=0}^n  \int_{\bra v \ket \geq \frac{1}{\sqrt{\alpha}}} 
	    \frac{ \bra v \ket^{2q + \gamma} \  \alpha^{aq}}{\Gamma(aq+1)} \; f(t,v) \; dv \\ \\
    
       & \geq & \frac{1}{\alpha^{\gamma/2}} \displaystyle \sum_{q=0}^n 
       \displaystyle \int_{\bra v \ket \geq \frac{1}{\sqrt{\alpha}}} \frac{ \bra v \ket^{2q} \  \alpha^{aq}}{\Gamma(aq+1)}  \; f(t,v) \; dv \\ \\

       &=& \frac{1}{\alpha^{\gamma/2}}  \left( 
           \displaystyle \sum_{q=0}^n \displaystyle \int_{\R^d} \frac{ \bra v \ket^{2q} \  \alpha^{aq}}{\Gamma(aq+1)}  \;  f(t,v) \; dv 
           
           \; - \; \displaystyle \sum_{q=0}^n \displaystyle \int_{\bra v \ket < \frac{1}{\sqrt{\alpha}}} 
            \frac{ \bra v \ket^{2q} \  \alpha^{aq}}{\Gamma(aq+1)}  \; f(t,v) \; dv  \right) \\ \\

        &\geq& \frac{1}{\alpha^{\gamma/2}}
            \left( \cE^n_a(t) \; - \; \displaystyle \sum_{q=0}^n \displaystyle \int_{\R^d} 
            \frac{ \alpha^{-q} \  \alpha^{aq}}{\Gamma(aq+1)}  \;  f(t,v) \; dv  \right) \\ \\
            
        &\geq& \frac{1}{\alpha^{\gamma/2}} \left(\cE^n_a(t) \; - \;  
          m_0 \; \displaystyle \sum_{q=0}^{\infty}  
		\frac{\alpha^{q(a-1)}}{\Gamma(aq+1)}\right) \\ \\

         &>& \frac{1}{\alpha^{\frac{\gamma}{2}}}\,\cE^n_a(t) \; - \; 
              \frac{1}{\alpha^{\frac{\gamma}{2}}} \, m_0 \; e^{\alpha^{a-1}}.
        \end{array}
\end{equation}

Therefore, applying inequality \eqref{prop-lower} to \eqref{prop-minus i} yields the following linear differential inequality for the partial sum $\cE^n_a$
\begin{equation*}
	\frac{d}{dt} \cE^n_a(t) \, \leq \,  - \,\frac{K_1}{2\, \alpha^{\frac{\gamma}{2}}} \cE^n_a (t)
		 \, + \, \frac{K_1 \, m_0 \, e^{\alpha^{1-a}}}{ 2 \alpha^{\frac{\gamma}{2}}}
		\, +   \cK_{0}.
\end{equation*}
Then, by the maximum principle for ordinary differential inequalities, 
\begin{equation*}
\begin{array}{lcl}
	 \cE^n_{2/s} (t) =  \cE^n_a(t) & \leq & \displaystyle  M_0 \, + \, \frac{2 \, \alpha^{\gamma/2}}{K_1} 
		\left(  \frac{K_1 \, m_0 \, e^{\alpha^{1-a}}}{ 2 \alpha^{\frac{\gamma}{2}}}
		\, + \,  \cK_{0}.\right) \\ 

	& = &  M_0\,  +\,  m_0  \, e^{\alpha^{1-a}}  \, + \, \frac{2 \, \alpha^{\gamma/2}}{K_1}  \, \cK_0 \,  \\

	&\leq&  4M_0,
\end{array}
\end{equation*}
provided that   $\alpha=\alpha_1$ is chosen sufficiently small so that 
\begin{equation}\label{alpha-choice}
   m_0  \, e^{\alpha_1^{1-a}}  \, + \, \frac{2 \, \alpha_1^{\gamma/2}}{K_1}  \, \cK_0  \, < \, 3 M_0\, .
 \end{equation}
 which is possible since $a>1$.

In conclusion, if $q_0$ is chosen according to \eqref{q04}, and hence depending only on the initial data, initial Mittag-Leffler moment, $\gamma$ and $A_2$ from \eqref{eq-nc}, and if $\alpha = \min \{ \alpha_0, (\ln 2)^{1/\alpha}, \alpha_1\}$, from \eqref{alpha-choice}, we have that the {\it strict} inequality $\cE^n_a(t) < 4M_0$ holds on the {\it closed} interval $[0, T_n]$ uniformly in $n$.  Therefore,  invoking the global continuity of $\cE^n_a(t) $ once more,  the set of time $t$ for $\cE^n_a(t) < 4M_0$   holds  on a slightly larger half-open time interval  $[0, T_n+\mu)$, with $\mu >0$. This would contradict maximality of the definition of $T_n$, unless $T_n = +\infty$. Hence, we conclude that $T_n = +\infty$ for all $n$. Therefore, we in fact have that
\begin{align*}
\cE^n_a(\alpha, t) < 4M_0, \quad \mbox{for all} \; t\geq 0, \;\; \mbox{for all} \, \, n \in \N.
\end{align*}
Thus, by letting $n \rightarrow +\infty$, we conclude that 
$\cE^\infty_a(\alpha, t) < 4M_0$  for all $t\geq 0$. That is,
\begin{equation}\label{final-prop1}
	\displaystyle\int_{\R^d}  f(t,v) \;\; \cE_{2/s} (\alpha^{2/s} \, \bra v \ket^2 ) \; dv \,
		    < 4M_0, \quad \mbox{for all} \,\, t \ge 0.
\end{equation}

Estimate \eqref{final-prop1} shows that  the solution of the Boltzmann equation with finite initial Mittag-Leffler moment of order $s$ and rate $\alpha_0$, will propagate Mittag-Leffler moments with the same order $s$ and rate $\alpha$ satisfying $\alpha = \min \{ \alpha_0, (\ln 2)^{1/\alpha}, \alpha_1\}$. This concludes the proof part{\bf(b)} of Theorem~\ref{thm}. 
\end{proof}

Part{\bf (a)} of Theorem~\ref{thm} concerns the generation of Mittag-Leffler or exponential moments. This is proven in the next section.

\section{Proof of exponential moments' generation } \label{sec-exp}

\begin{proof}[Proof of Theorem \ref{thm} (a)]
\noindent Notation and strategy are similar to those in the proof of Theorem \ref{thm} (b), contained in Section \ref{sec-ml}. The goal is to find a positive and bounded real valued number $\alpha$ such that the solution $f(v,t)$ of the Boltzmann equation will have an exponential moment,  of order $\gamma$ and rate $\alpha\min\{ t,1\}$, generated for every positive time $t$, from the fact that the initial data $f_0(v)$ has finite energy given by $M_0^*:= m_2(0)$. 

The proof works with the exponential forms of order $\gamma$. From this viewpoint, the difference
with respect to the propagation of Mittag-Leffler moments result obtained in the previous section is that
the propagation result had to be established for every order $s \in (0, 2)$,
while now the generation of Mittag-Leffler moments  of order $s$  and rate $\alpha$ implies generation of such moments for all smaller orders $0<s$. Hence, it suffices to consider just the order $\boldsymbol{s=\gamma}$. 

First for an arbitrary positive and bounded number $\alpha$, 
we  denote the $n$-th partial sum of the exponential moment of order $\gamma$ by $E^n_\gamma(\alpha t, t)  $ and the corresponding one in which polynomial moments are shifted by $\gamma$ by 
$I^n_{\gamma,\gamma}(\alpha t, t) $, that is
\begin{align}\label{E}
 E^n_\gamma (\alpha t, t) \ &=\  \displaystyle \sum_{q=0}^n \frac{m_{\gamma q}(t) \; (\alpha t)^q}{\Gamma(q+1)} \ =\ \displaystyle \sum_{q=0}^n \frac{m_{\gamma q}(t) \; (\alpha t)^q}{q!} \\
 \cI^n_{\gamma,\gamma} (\alpha t,t) \ &= \ \displaystyle \sum_{q=0}^n \frac{m_{\gamma q + \gamma}(t) \; (\alpha t)^q}{\Gamma(q+1)} \ =\ 
 \displaystyle \sum_{q=0}^n \frac{m_{\gamma q + \gamma}(t) \; (\alpha t)^q}{q!} \, . \label{I}
 \end{align}
The form $E^n_\gamma(\alpha t, t)  $ is the exponential moment of order $\gamma$ with rate $\alpha$ of the probability density $f$ in the  Mittag-Leffler representation.

\noindent Define the time $T^*_n$ as follows
\begin{equation}\label{T_ng}
 T^*_n:= \min \left\{ 1, \;\;  \sup \left\{ t \geq 0 \;|\; E^n_\gamma(\alpha \tau , \tau) <  4M^*_0, \quad \mbox{for all}\,\, \tau \in [0, t) \,\,    \right\} \right\}.
\end{equation}
$T^*_n$ is well defined where now the constant $M^*_0$ is the sum of the initial conserved  mass and energy, i.e. $ M^*_0:=M^*_0(t)=\int f(v,t)\langle v\rangle^2 dv = \int f_0(v) \langle v\rangle^2 dv $ as in  the  initial condition for the generation of Mittag-Leffler moments estimate
 \eqref{eq-gener}.  Since moments are uniformly in time generated for the hard potential case, even for angular non-cutoff regimes (see \cite{we96}), then every finite sum $\cE^n_a(\alpha t, t)$ is well defined and continuous in time.  Note that for $t=0$, we have that $E^n_\gamma (\alpha 0, 0) = m_0 <  4M^*_0$. Then, as  in the previous case, continuity in time of partial sums $\cE^n_a(\alpha t, t)$ implies that $\cE^n_a(\alpha t, t) < 4M^*_0$ 
holds for $t$ on some positive time interval  $[0, t^*_n)$, which implies that $T^*_n >0$.   In addition,  the definition \eqref{T_ng} implies that $T^*_n\le 1$ for all $n \in \N$.

As we did in the previous section for the proof of propagation of Mittag-Leffler moments, we search for an ordinary differential inequality for $E^n_\gamma(\alpha t, t) $, depending only on data parameters and on $ \cI^n_{\gamma,\gamma} (\alpha t,t)$, for a positive and bounded real valued $\alpha$ to be found and characterized. 

To this end, we start by computing  
\begin{align}\label{gen-4sums0}
\displaystyle \frac{d }{dt} E^n_\gamma(\alpha t, t) 
    \; &=  \; \alpha  \sum_{q=1}^{n} \frac{m_{\gamma q}(t)\,  (\alpha t)^{q-1} }{(q-1)!}\; +\; \displaystyle\sum_{q=0}^{n} \displaystyle \frac{m'_{\gamma q}(t)\,  (\alpha t)^q}{q!}  \\
        &= \; \alpha  \sum_{q=1}^{n} \frac{m_{\gamma q}(t)\, (\alpha t)^{q-1} }{(q-1)!}   \; + \; \displaystyle\sum_{q=0}^{q_0-1} \displaystyle \frac{m'_{\gamma q}(t)\, (\alpha t)^q}{q!}
        \;  +\; \displaystyle\sum_{q=q_0}^{n} \displaystyle \frac{m'_{\gamma q}(t)\, (\alpha t)^q}{q!}  \, ,      \nonumber
\end{align}
where index $q_0$ will be fixed later. The first sum in this identity is reindexed by from $q-1$ to $q$ and estimated by $I^n_{\gamma, \gamma}(\alpha t, t)$ (defined in \eqref{I}), as follows 
\begin{equation*} 
 \sum_{q=0}^{n-1} \frac{m_{\gamma q + \gamma}(t)\, (\alpha t)^{q} }{q!}  \le  \sum_{q=0}^{n} \frac{m_{\gamma q + \gamma}(t)\, (\alpha t)^{q} }{q!}  =   I^n_{\gamma, \gamma}(\alpha t, t).
\end{equation*}

Next, replacing the term $m'_{\gamma q}(t)$ by the upper bound in the ordinary differential inequality (\ref{thm-b}) just on the sums starting from $q_0$, for $\alpha>0$, and for

\begin{equation}\label{kq*}
k_{q^*}:= \lfloor\frac q4 - \frac 1{\gamma} + \frac{3}{2}\rfloor  \  := \  \mbox{ integer part of}\  \frac q4 - \frac1{\gamma} + \frac{3}{2} \, , 
\end{equation}
\begin{align}\label{gen-4sums}
\displaystyle \frac{d }{dt} E^n_\gamma(\alpha t,t)  \; &\leq \; 
 \alpha  I^n_{\gamma, \gamma}(\alpha t, t)\; +\; \displaystyle\sum_{q=0}^{q_0-1} \displaystyle \frac{m'_{\gamma q}(t)\, (\alpha t)^q}{q!} \nonumber\\
   	\qquad \, &- \, K_1 \displaystyle \sum_{q=q_0}^{n} \frac{m_{\gamma q + \gamma}(t)\,  (\alpha t)^q}{q!}   
	 \, +\,  K_2  \displaystyle \sum_{q=q_0}^{n} \frac{m_{\gamma q}(t)\,  (\alpha t)^q}{q!} \\
	\qquad &+ \, K_3 \displaystyle \sum_{q=q_0}^{n} 
	    \frac{\varepsilon_{{\gamma q/2}}  \frac{\gamma q}2 (\frac{\gamma q}2-1) \; (\alpha t)^q}{q!} 	    
	     \displaystyle \sum_{k=1}^{k_{q_*}} \binom{\frac q2 - \frac 2{\gamma} }{k-1} \nonumber \\
	&\ \qquad  \qquad \qquad  \left( (m_{2\gamma k+\gamma}(t)\,  m_{\gamma q-2\gamma k}(t)\, 
	      +\,m_{2\gamma k}(t)\,  m_{\gamma q-2\gamma k+\gamma}(t)\, \right) \nonumber \\
	      	&\     \qquad\nonumber   \\
	\qquad \; &=: \;    \alpha  \cI^n_{\gamma, \gamma}(\alpha t,t) \,+ \, S_0 	
- K_1 \, S_1 \, + \, K_2 \, S_2  \,+ \,  K_3 \,S_3 \, . \nonumber
	\end{align}

We stress the positive constant $K_1 = A_2\, C_\gamma$ depends only on the collision cross section with $A_2$ defined in \eqref{part1}, and $C_\gamma$ only depending on $0<\gamma\le 1$.  In the sequel, we will estimate the terms in \eqref{gen-4sums} to show that the negative one is of higher order uniformly in time $t$, for a choice of $\alpha$ and $q_0$ that depend only on the initial and collision kernel data. 

The term $S_0$ can be bounded by terms that depends on the initial data and the parameters of the collision cross section.
Indeed,  as was the case for the propagation estimates,  from Lemma \ref{prop-genpoly}, setting $r=\gamma$ in \eqref{bBrp}, the generated polynomial moments can be estimated  by
\begin{align}\label{gen-s00}
m_{\gamma q}(t) & \leq \bB_{\gamma q} \; \max_{t>0}\{1, t^{-q} \}\qquad\mbox{and} \\
m'_{\gamma q}(t)\,  &\leq  B_{\gamma q}   m_{\gamma q}(t)\, \leq \,  B_{\gamma q}   \, \bB_{\gamma q} \; \max_{t>0}\{1, t^{- q}\}  \nonumber \\
\end{align}
where  the constant ${\bf B}_{\gamma q}$, now from \eqref{bBrp},  also depends on $m_2(0)$, $\gamma$, $q$  and $A_2$ from condition \eqref{eq-nc}. 
Next, for $q_0$ fixed, to be chosen later, set
\begin{align}
c^*_{q_0}: = \max_{q \in \{ 0, \dots, q_0-1 \}} \{\bB_{\gamma q}, \,  B_{\gamma q} \, \bB_{\gamma q}\}\, , 
\end{align}
and then, both the $2q$-moments and its derivatives are  controlled by $c^*_{q_0}$  as follows
\begin{align} 
m_{\gamma q}(t) , m'_{\gamma q}(t)   \ \leq c^*_{q_0} \; \max_{t>0}\{1, t^{- q}\}, \quad \mbox{for all } \, q \in \{0,\dots,q_0-1\}\, . 
\end{align}

\medskip

Thus we can estimate $S_0$, for a fixed $q_0$ to be defined later,  by 
\begin{align}
  S_0 &:= \displaystyle \sum_{q=0}^{q_0 -1} \frac{m'_{\gamma q}(t)\,  (\alpha t)^q}{q!}    \nonumber \\
	& \leq c^*_{q_0} \; \max_{t>0}\{1, t^{- q}\} \,\displaystyle \sum_{q=0}^{q_0 -1}
		 \frac{ (\alpha t)^q } {q!} \nonumber \\      
	&\leq   c^*_{q_0} \; \max_{t>0}\{t^q, 1\} \, \displaystyle \sum_{q=0}^{q_0 -1}
		 \frac{\alpha^q }{q!} \label{gen-s10} \\
           &\leq  c^*_{q_0} e^{\alpha} \leq  2 \, c^*_{q_0},   \label{gen-s1}
\end{align}
uniformly in $t\in [0,T^*_n] \subset [0,1]$,  for any $\alpha \leq \ln 2$. To obtain inequality \eqref{gen-s10} we used that $t \leq T^*_n \leq 1$.

 The sum $S_2$ is a part of the partial sum $E^n_\gamma$, hence
\begin{align}\label{gen-s3}
S_2 := \displaystyle \sum_{q=q_0}^n \frac{m_{\gamma q} (\alpha t)^q}{q!} \leq E^n_{\gamma} (\alpha t, t).
\end{align}

 The sum $S_1$ needs to be bounded from below because of the negativity of the term $K_1\, S_1$.
 To this end,   using again the time dependent  estimates for moments  from   Proposition \ref{prop-genpoly},   the estimate from below follows for $t\in (0, T^*_n] \subset (0,1]$ as
 \begin{equation}\label{gen-s2}
  \begin{array}{lcl}
    S_1 &:=&   \displaystyle\sum_{q=q_0}^{n} \frac{m_{\gamma q + \gamma}(t)\,  (\alpha t)^q}{q!}  \, =\, 
      I^n_{\gamma,\gamma} (\alpha t,t) - \displaystyle \sum_{q=0}^{q_0 -1} \frac{m_{\gamma q+\gamma} (\alpha t)^q}{q!} \\
        
      &\geq& I^n_{\gamma,\gamma}(\alpha t,t) - c^*_{q_0}  \displaystyle \sum_{q=0}^{q_0 -1} 
	  \frac{ \max_{0<t\le1}\{1, t^{-(\gamma q+\gamma)/\gamma}\} (\alpha t)^q}{q!} \\
	  
      &\geq& I^n_{\gamma,\gamma}(\alpha t, t) -  c^*_{q_0}  \displaystyle \sum_{q=0}^{q_0 -1} 
       \frac{  t^{-q-1}  (\alpha t)^q}{q!} \\
              
      &=& I^n_{\gamma,\gamma}(\alpha t, t) - \frac{ c^*_{q_0}}{t}  \displaystyle \sum_{q=0}^{q_0 -1} 
       \frac{ \alpha^q}{q!} \\ \\

       &\geq& I^n_{\gamma,\gamma}(\alpha t, t) - \frac{ c^*_{q_0}}{t} e^{\alpha} \\ \\

   &\geq& I^n_{\gamma,\gamma}(\alpha t, t) - \frac{2  c^*_{q_0}}{t} . 
       
   \end{array}
 \end{equation}

The estimate for the  double sum term  in $S_3$  uses an analogous treatment to the one in the previous section to obtain  Mittag-Leffler moment's propagation.  More precisely,  set $S_3 :=  S_{3,1} +  S_{3,2} $,  and we make use of the identity \eqref{Gproperty} written in the following format 
 \begin{equation} 
 \Gamma(2k+1) \Gamma(q-2k+1) =  B(2k+1, q-2k+1) \, \Gamma(q+2)
 \end{equation} 
 to obtain
\begin{align}\label{ss31}
 S_{3,1}  &:= \displaystyle \sum_{q=q_0}^{n} \varepsilon_{{\gamma q}/2} 
		\frac{\gamma q}{2} \left(\frac{\gamma q}{2} - 1\right)       
		\displaystyle \sum_{k=1}^{k_{q_*}} \binom{\frac{q}{2} - \frac{2}{\gamma}}{k-1} 
	 	\frac{m_{2\gamma k+\gamma}(t) \, (\alpha t)^{2k}}{\Gamma(2k+1)}   
		\frac{m_{\gamma q-2\gamma k}(t) \, (\alpha t)^{q-2k}}{\Gamma(q-2k+1)} \nonumber \\
	&\ \qquad   \qquad \qquad \qquad \qquad  \qquad \qquad B(2k+1, q-2k+1) \frac{\Gamma(q+2)}{\Gamma(q+1)}  \\ 
	& \;  \leq  \; \varepsilon_{\gamma q_0/2} \displaystyle \sum_{q=q_0}^{n}  (q+1) 
		\frac{\gamma q}{2} \left(\frac{\gamma q}{2} - 1\right)          
	         \left( \displaystyle \sum_{k=1}^{k_{q_*}} 
		 \frac{m_{2\gamma k+\gamma}(t) \; (\alpha t)^{2k}}{\Gamma(2k+1)}
           	\frac{m_{\gamma q-2\gamma k}(t) \; (\alpha t)^{q-2k}}{ \Gamma(q-2k +1 )}  \right)          \nonumber \\
	& \qquad\qquad\qquad\qquad\qquad\qquad 
 \left( \displaystyle \sum_{k=1}^{k_{q_*}}   \binom{\frac{q}{2} - \frac{2}{\gamma}}{k-1}  \;  B(2k+1, q-2k+1) \right)  \, .\nonumber
\end{align}

\noindent The last inequality was obtained via the inequality
$\sum_k a_k b_k \leq \sum_k a_k \; \sum_k b_k$, and the fact that $\varepsilon_q$ decreases in $q$. 
Again,  using the estimate of Lemma~\ref{lem-sumB2}, the sum of the Beta functions  is bounded by $C q^{-3}$, with $C$ a uniform constant independent of $q$. Therefore, 
\begin{align}\label{ss01}
 (q+1) \frac{\gamma q}{2} \left(\frac{\gamma q}{2} - 1\right)   & \left( \displaystyle \sum_{k=1}^{k_{q_*}}   \binom{\frac{q}{2} - \frac{2}{\gamma}}{k-1}  \;  B(2k+1, q-2k+1) \right) \nonumber  \\
 &\ \le \,  (q+1) \frac{\gamma q}{2} \left(\frac{\gamma q}{2} - 1\right) q^{-3} \ \le \  C_\gamma \, ,
\end{align}
uniformly in $q$. Then,  estimating  the right hand side of \eqref{ss31}  by the estimate \eqref{ss01} just above, yields
\begin{equation}\label{ss1}
 S_{3,1} \leq K_3 \; C_\gamma \;  \varepsilon_{\gamma q_0/2} \displaystyle \sum_{q=q_0}^n
  \left( \displaystyle \sum_{k=1}^{k_{q_*}}  \frac{m_{2\gamma k+\gamma}(t) \; (\alpha t)^{2k}}{\Gamma(2k+1)}
              \frac{m_{\gamma q-2\gamma k}(t) \; (\alpha t)^{q-2k}}{ \Gamma(q-2k +1)}  \right).    
\end{equation}

Finally, as was the case for the propagation estimates in the previous section, changing the order of summation in the right hand side of  \eqref{ss1}
yields a  control   by a factor  $E^n_\gamma(\alpha t,t) \,  \cI^n_{\gamma,\gamma}(\alpha t,t) $ as follows.  
Recalling the definition   of $k_{q_*}$ from \eqref{kq*}, and evaluating it for $n$ instead of $q$ yields
\begin{align*}
 S_{3,1} 
& \; \leq \;   C_\gamma \, \varepsilon_{\gamma q_0/2}
         \displaystyle \sum_{k=0}^{\left[ \frac{n}{4} + \frac{3}{2} - \frac{1}{\gamma} \right]}  \sum_{q=\max\{q_0, 4k-2\}}^{n} 
          \frac{m_{2\gamma k+\gamma} \; (\alpha t)^{2k}}{\Gamma(2k+1)}
           \frac{m_{\gamma q-2\gamma k} (\alpha t)^{q-2k}}{\Gamma(q - 2k +1)} \\
& \; = \;  C_\gamma \, \varepsilon_{\gamma q_0/2}  \displaystyle \sum_{k=0}^{\left[ \frac{n}{4} + \frac{3}{2} - \frac{1}{\gamma}\right]} 
           \frac{m_{2\gamma k+\gamma}(t) \; (\alpha t)^{2k}}{\Gamma(2k+1)} 
           \left(  \sum_{q=\max\{q_0, 4k-2\}}^{n}  \frac{m_{\gamma q - 2\gamma k}(t)  (\alpha t)^{q-2k}}{\Gamma(q - 2k +1)} \right) \\
& \; \leq \;  C_\gamma \, \varepsilon_{\gamma q_0/2} \;   \displaystyle 
          \sum_{k=0}^{\left[ \frac{n}{4} + \frac{3}{2} - \frac{1}{\gamma}\right]} 
           \frac{m_{2\gamma k+\gamma}(t) \; (\alpha t)^{2k}}{\Gamma(2k+1)}    \, E^n_\gamma (\alpha t,t)\\ \\
& \; \leq \;  C_\gamma \, \varepsilon_{\gamma q_0/2} \; \cI^n_{\gamma,\gamma}(\alpha t,t) \, E^n_\gamma (\alpha t,t) \, . 
\end{align*}

\noindent Analogous estimate can be obtained for  $S_{3,2}$, so overall we have
\begin{equation}\label{gen-s4}
  S_3 \ \le \  2 C_\gamma  \varepsilon_{\gamma q_0/2} \;  \cI^n_{\gamma,\gamma}(\alpha t,t) \; E^n_\gamma (\alpha t,t)  \; .
\end{equation}

\noindent  Therefore, combining  estimates (\ref{gen-s1}), \eqref{gen-s2}, \eqref{gen-s3} and (\ref{gen-s4}) with \eqref{gen-4sums} yields the following differential inequality 
for  $E^n_\gamma = E^n_\gamma (\alpha t,t)$ depending on $\cI^n_{\gamma, \gamma} = \cI^n_{\gamma, \gamma}(\alpha t,t)$,
\begin{equation*}\label{ode-prop}
 \frac{d}{dt} E^n_{\gamma}  \leq 2 c^*_{q_0} + 
    \left(-K_1 \; \cI^n_{\gamma, \gamma}   + K_1 \; \frac{2 \,c^*_{q_0}}{t} 
        + K_2 \; E^n_{\gamma} 
        + 2 \varepsilon_{\gamma q_0/2} C_\gamma\,  K_3  E^n_\gamma  \; \cI^n_{\gamma, \gamma}\right)     
    + \alpha \cI^n_{\gamma, \gamma}
\end{equation*}
This inequality is the analog to the one in \eqref{prop-ei} for the propagation argument.
Since the  partial sum $E^n_\gamma(\alpha t,t)$ is bounded by $4M^*_0$ on the interval $[0,T^*_n]$, uniformly in $n$ and $T^*_n \le1$, then 
the right hand side of the above inequality is controlled by
\begin{equation*}
  \frac{d}{dt} E^n_{\gamma} (\alpha t,t)
	\leq   -\cI^n_{\gamma, \gamma}(\alpha t,t) \Big( K_1  - 8 M^*_0 \, \varepsilon_{\gamma q_0/2} C_\gamma  K_3  - \alpha \Big)
  + 4M^*_0 \, K_2  +  \frac{2 \, K_1 \,  c^*_{q_0}}{t} + 2  c^*_{q_0}.
\end{equation*}

Next, since  $t \leq T^*_n \leq 1$, then $t^{-1} \geq 1$, so the above estimate is further bounded by
\begin{equation*}
  \frac{d}{dt} E^n_{\gamma}  (\alpha t,t)
	\leq   -\cI^n_{\gamma, \gamma}(\alpha t,t) \Big( K_1  - 8 M^*_0\, \varepsilon_{\gamma q_0/2} C_\gamma K_3   - \alpha \Big)
	  +  \frac{\cK_{q_0}}{t}.
\end{equation*}
with $0 < \cK_{q_0}= 2 c^*_{q_0} +   4M^*_0 K_2 + 2 K_1 c^*_{q_0}$ only depending on data parameters, including $q_0$, independent of $n$.

Finally, since $\varepsilon_{\gamma q_0/2}$ converges to zero as $q_0$ goes to infinity, we can choose large enough $q_0$ 
  and small enough $\alpha$  so that  
 $b(\cos{\theta})$ satisfies \eqref{eq-nc} with $\beta = 2a -2$ ), 
 \begin{equation}\label{gg4} 
 K_1 \, - \, 8\varepsilon_{q_0} \; q_0^{2-a} \, K_3  - \alpha >   \frac{K_1}2 \, ,
   \end{equation} 
which yields 
 \eqn \label{ei}
	\displaystyle \frac{d }{dt} \cE^n_a(\alpha_1 t,  t) 
	  \, \leq - \, \frac{K_1}{2} \, \cI^n_{a, \gamma}(\alpha t,  t)   \, + \,  \frac{\cK_{q_0}}{t} \; .
\eeqn

Therefore, the  final step consists  in finding  a lower bound for $\cI^n_{a,\gamma}(\alpha t,  t)$ in terms of $\cE^n_a(\alpha t,  t)$ as follows
\begin{align}\label{lower}\nonumber
I^n_{\gamma, \gamma} (\alpha t,  t) 
& \; = \; \displaystyle \sum_{q=0}^n \frac{m_{\gamma(q + 1)}(t)  \; (\alpha t)^q}{q!}
	\; = \; \displaystyle\sum_{q=1}^{n+1} \frac{m_{\gamma q}(t) \; (\alpha t)^q}{q!} \frac{q}{\alpha t} \\ 
&\; \geq \; \frac{1}{\alpha t} \displaystyle\sum_{q=3}^{n} \frac{m_{\gamma q}(t) \; (\alpha t)^q}{q!}       
	\;= \;\frac{E^n_\gamma(t, \alpha t) - M^*_0}{\alpha t}.
\end{align}

Combining (\ref{ei}) and (\ref{lower}) yields
\begin{align*}
 \frac{d}{dt} E^n_{\gamma}(\alpha t,  t)  \leq -\frac{1}{t} \left(\frac{K_1(E^n_\gamma - M^*_0)}{2 \alpha} - \cK_{q_0}\right)  =  -\frac{K_1}{2\alpha t} \left( E^n_\gamma - M^*_0 -  \frac{2 \alpha}{K_1} \cK_{q_0}\right) .
\end{align*}
Then choosing a small enough  $\alpha$  such that 
  \begin{equation}\label{g-alpha} 
 M^*_0 + \frac{2\alpha}{K_1}{\cK_{q_0}} < 2 M^*_0  \qquad \mbox{or, equivalently,} \qquad      \alpha <  \frac{K_1  M^*_0}{2\cK_{q_0}}   \, ,
   \end{equation}
yields  
 \begin{equation}\label{qq5}
 \frac{d}{dt} E^n_{\gamma}(\alpha t,  t)  \leq -\frac{K_1}{2\alpha t} \left( E^n_\gamma(\alpha t,  t) - 2  M^*_0\right).
 \end{equation}

 Then, by a comparison argument,  whenever $E^n_\gamma (\alpha t,  t)  > 2 M^*_0$,  we have $\frac{d}{dt} E^n_{\gamma} < 0$,  and so $E^n_\gamma (\alpha t,  t) $ decreases in $t$. 
 Since at initial time the partial sum is less that the threshold, i.e. $E^n_\gamma(0,0)=m_0 <  2M^*_0$  and since it is continuous for all  times,  we have that the {\it strict} inequality $E^n_\gamma (\alpha t,t)  \leq 2 M^*_0 < 4  M^*_0$  holds uniformly on the { closed} interval $[0, T^*_n]$. By continuity of the partial sum, this strict inequality $E^n_\gamma(\alpha t, t) < 4 M^*_0$ then holds on a slightly larger interval, which would contradict maximality of $T^*_n$ from the definition \eqref{T_ng}, unless $T^*_n=1$. Hence, we conclude that  $T^*_n=1$ for all $n$.

Therefore, we in fact have that
\begin{align*}
E^n_\gamma(\alpha t, t) < 4M^*_0, \quad \mbox{for all} \; t\in [0, 1] \;\; \mbox{for all} \, \, n \in \N.
\end{align*}
Thus, by letting $n \rightarrow +\infty$, we conclude that 
$E^\infty_\gamma(\alpha t, t) < 4M^*_0$  for all $t\in [0,1]$. That is, 
\begin{equation}\label{final-gen1}
	\displaystyle\int_{\R^d}  f(t,v) \;\; \cE_{2/\gamma} ((\alpha t)^{2/\gamma} \, \bra v \ket^2 ) \; dv \,
		    < 4M_0, \quad \mbox{for all} \,\, t \in [0,1].
\end{equation}

To finalize the proof, first set   $\alpha=\min\{ \ln 2,  \alpha_1 \}$, from \eqref{gen-s1} and with $\alpha_1$  satisfying condition \eqref{g-alpha}  that depends on the initial data, $\gamma$, the collisional kernel and  $A_2$ from the integrability condition \eqref{eq-nc}. This $\alpha$ is a positive and bounded real number.
 
Then, note that the above inequality implies that at the time $t=1$, the Mittag-Leffler moment  of order $\gamma$ and rate $\alpha t = \alpha$ is finite. Now, starting the argument from $t=1$ on, we bring ourselves into the setting of the propagation and conclude that for $t \ge 1$, the  Mittag-Leffler moment of the same order $\gamma$ and potentially smaller $\alpha$ than the one found on time interval $[0,1]$, remain uniformly bounded for all $t \ge 1$.

In conclusion,
\begin{align}
\displaystyle\int_{\R^d}  f(t,v) \;\; \cE_{2/\gamma} ((\alpha t)^{2/\gamma} \, \bra v \ket^2 ) \; dv \,
		    < C, \quad \mbox{for all} \,\, t \in [0,1],
\end{align}
and 
\begin{align}
\displaystyle\int_{\R^d}  f(t,v) \;\; \cE_{2/\gamma} (\alpha^{2/\gamma} \, \bra v \ket^2 ) \; dv \,
		    < C, \quad \mbox{for all} \,\, t  \geq 1.
\end{align}
Therefore, we conclude that for all  $t \geq 0$, we have
\begin{align}
\displaystyle\int_{\R^d}  f(t,v) \;\; \cE_{2/\gamma} ((\alpha \min\{1, t\})^{2/\gamma} \, \bra v \ket^2 ) \; dv \,   < C.
\end{align}

 In particular, this asserts that the solution of the Boltzmann equation with an initial mass and energy, will develop Mittag-Leffler moments,  or equivalently, exponential high energy tails of order $\gamma$ with     rate 
 $r(t)=\alpha \min\{t,1\}$. Therefore the proof of Theorem~\ref{thm} is now complete. 

\end{proof}


\appendix
\section{}\label{appx}

We gather technical results used throughout this manuscript.  The first two lemmas focus on elementary polynomial inequalities that will be used to derive ordinary differential inequalities for polynomial moments in Section \ref{sec-ode}.

 \begin{lemma}[Polynomial inequality I]\label{trick}
  Let $b \leq a \leq \frac{s}{2}$. Then for any $x, y \geq 0$
  \begin{equation}\label{eq-trick}
   x^a y^{s-a}  + x^{s-a} y^a  \; \leq \;  x^b y^{s-b}  +  x^{s-b} y^b.
  \end{equation}
 \end{lemma}

\begin{remark} This lemma is useful for comparing products of moments. Namely, as its consequence, we have that for a fixed $s$, the sequence $\{m_{k} \; m_{s-k}\}_k$ is decreasing in $k$, for $k = 1, 2, ... , \lfloor s/2 \rfloor:=$ Integer Part of $s/2$.   For example, if $s\geq 4$, then $m_2 m_{s-2} \leq m_1 m_{s-1}$.
\end{remark}
\begin{prf}
Note that $a, b$ and $s$ satisfy $a-b \geq 0$ and $s-a-b \geq 0$. Therefore
\begin{equation*}
	\left( y^{a-b} - x^{a-b} \right) x^b y^b \left( y^{s-a-b} - x^{s-a-b} \right) \geq 0,
\end{equation*}
which is easily checked to be equivalent to the inequality \eqref{eq-trick}.
\end{prf}

\begin{lemma}[Polynomial inequality II, Lemma 2 in \cite{bogapa04}]\label{lm-comb}
Assume $p>1$, and let $k_p = \lfloor (p+1)/2 \rfloor$. Then for all $x, y >0$ the following inequalities hold
\eqnn
	\sum_{k=1}^{k_p -1} \binom{p}{k} (x^k y^{p-k} \,+\, x^{p-k} y^k) 
	\, \leq \, 
	(x+y)^p \, - \, x^p \, - \, y^p
	\, \leq \, 
	\sum_{k=1}^{k_p} \binom{p}{k} (x^k y^{p-k} \, + \, x^{p-k} y^k). 
\eeqnn
\end{lemma}

\begin{remark}\label{rk-poly}
Using this lemma, it is easy to see a rough, but useful estimate
\begin{align}\label{rk-poly2} 
\sum_{k=0}^{k_p} \binom{p}{k} (x^k y^{p-k} \,+\, x^{p-k} y^k) 
	\, \leq \, 2 (x + y)^p.
\end{align}
\end{remark}

Next, we recall the basic definitions and  properties of Gamma  $\Gamma(x)$ and Beta $B(x,y)$ functions that are useful for the next estimates. 
They are defined via
\begin{align}
 \Gamma(x) \, = \, \int_0^\infty t^{x-1} \, e^{-t} \, dt,  \qquad \mbox{and} \qquad
 B(x,y) \, = \, \int_0^1 t^{x-1} \, (1-t)^{y-1} \, dt,
\end{align}
respectively.  Two fundamental properties of these well-know functions are 
\begin{align} \label{Gproperty}
\Gamma(x+1) = x \, \Gamma(x),\qquad \mbox{and} \qquad  B(x,y) = \frac{\Gamma(x) \, \Gamma(y)}{\Gamma(x+y)}
\end{align}

 The following classic result for estimates of generalized Laplace transforms  will be needed  to estimate the combinatoric sums of Beta functions to be shown in the subsequent Lemma~\ref{lem-sumB1}. 
 \begin{lemma}[]\label{sp}
  Let $0 < \alpha, R < \infty$, $g\in C([0,R])$ and $S \in C^1 ([0,R])$ be such that
  $S(0)=0$ and $S'(x) <0$ for all $x\in [0,R]$. Then for any $\lambda \geq 1$ we have
  \begin{equation*}   
      \displaystyle \int_0^R x^{\alpha -1} \; g(x) \; e^{\lambda S(x)} \; dx =
      \Gamma(\alpha) \; \left( \frac{1}{-\lambda \; S'(0)}\right)^\alpha \; (g(0) + o(1)).
  \end{equation*}
 \end{lemma}
The proof of this estimate is a direct application of the Laplace's method for asymptotic expansion of integrals that can be found in \cite{ol74}, page 81, Theorem 7.1.

\noindent The next two lemmas estimate  a combinatoric sum  of Beta functions. These estimates are inspired by the work in
Lemma 4 in \cite{bogapa04} and Lemma 3.3  in \cite{lumo12}. However, in our context, the arguments of Beta functions are shifted, so we compute exact decay rates for our situation. These estimates are crucial  to control the growth in $q$ of the ordinary differential inequality of partial sums of renormalized moments.

 The first lemma will be used for the proof of propagation of moments with $a=2/s$, while the second will be used for the generation of moments with $s = \gamma$.

 \begin{lemma}[First estimate on combinatoric sums of Beta Functions ] \label{lem-sumB1} 
  Let $q \geq 3$ and $k_q = \left[(q+1)/2\right]$. Then for any $a>1$ we have
  \begin{equation}\label{beta}
     \displaystyle \sum_{k=1}^{k_q} \displaystyle \binom{q - 2}{k-1}  B(ak+1, a(q-k)+1)
	\leq C_a \; \frac{1}{(aq)^{1+a}},
  \end{equation}
  
  where the constant $C_a$ depends only on $a$.
  
 \end{lemma}

\begin{prf}  Reindexing the summation from $k=1$ to $k=0$ by changing $k-1$ into $k$  and rearranging the integral forms defining  Beta functions, yields
\begin{align*}
 \displaystyle  & \sum_{k=1}^{k_q} \displaystyle \binom{q - 2}{k-1}  B(ak+1, a(q-k)+1)
	\\
	& = \displaystyle \sum_{k=0}^{k_q-1}   \binom{q - 2}{k}  \; B(a(k+1)+1, a(q-k-1)+1) \\
    & =  \displaystyle\frac{1}{2} \displaystyle \int_0^1 
           \displaystyle \sum_{k=0}^{k_q-1}  \binom{q - 2}{k} 
           \left( x^{a(k+1)} \; (1-x)^{a(q - k - 1)} \; + \; x^{a(q-k-1)} \; (1-x)^{a(k+1)} \right) \; dx \\
    & = \displaystyle\frac{1}{2} \displaystyle \int_0^1 x^a (1-x)^a 
          \displaystyle \sum_{k=0}^{k_{q-2}} \binom{q-2}{k}  
           \left( x^{ak} \; (1-x)^{a(q - 2 - k)} \; + \; x^{a(q-2-k)} \; (1-x)^{ak} \right) \; dx \\
           & = \displaystyle\frac{1}{2} \displaystyle \int_0^1 x^a (1-x)^a \displaystyle \sum_{k=0}^{k_p} \binom{p}{k}  
           \left( x^{ak} \; (1-x)^{a(p - k )} \; + \; x^{a(p-k)} \; (1-x)^{ak} \right) \; dx 
           \end{align*}
after setting $q-2 = p$ in the last integral. The estimate  \eqref{rk-poly2}, then yields
\begin{align*}
 \displaystyle   \sum_{k=1}^{k_q} \displaystyle \binom{q - 2}{k-1} & B(ak+1, a(q-k)+1) \\
 &\leq 
  \displaystyle \frac{1}{2} \displaystyle \int_0^1 x^a (1-x)^a  \; 2\; \left(x^a + (1-x)^a \right)^{p} \; dx \\
  &=
\displaystyle \int_0^1 x^a (1-x)^a \left(x^a + (1-x)^a \right)^{q-2} \; dx \\
     &= 2 \displaystyle \int_0^{1/2} x^a \; g(x) \; e^{q S(x)} \; dx,
\end{align*}
 where $g(x) = (1-x)^a \left( x^a + (1-x)^a \right)^{-2}$ and $S(x) = \log(x^a + (1-x)^a)$, for $x \in [0, 1/2]$.
Finally, applying  Lemma \ref{sp} for  these $g(x)$ and $S(x)$ as indicated, and noting that  $g(0) = 1$ and $S'(0) =-a$,  yields the desired estimate
\begin{equation}\label{Beta1}
 \displaystyle  \sum_{k=1}^{k_q} \displaystyle \binom{q - 2}{k-1}  B(ak+1, a(q-k)+1) \; \leq \;  C_a \; \Gamma(a+1) \; \left( \frac{1}{a q} \right)^{a+1}.
\end{equation}
\end{prf}

 \begin{lemma}[Second estimate on combinatoric sums of Beta Functions] \label{lem-sumB2}
 Let $0<s\leq 1$ and $q\geq 3$. Then, there exits a constant $C$, independent on $q$, such that
  \begin{equation}\label{lem-sumB2-0}
   \displaystyle \sum_{k=1}^{ 1 + k_{\frac{q}{2}-\frac{2}{s}} } \;
     \binom{\frac{q}{2}-\frac{2}{s} }{k-1} \; B(2k+1, \; q - 2k +1) \leq C \frac{1}{q^3}.
 \end{equation}  
 \end{lemma}

\begin{prf}
 \noindent First we note a simple property of binomial coefficients. For any integer $k \in \N_0$ and
 any real numbers $\tilde{a}, a \in \R$ that satisfy  $\tilde{a} \geq a \geq k$, 
 \begin{equation}\label{binom}
  \binom{a}{k} \leq \binom{\tilde{a}}{k}.
 \end{equation}
 \noindent This is easily proved by noting that the binomial coefficient $\binom{a}{k}$ 
 (and similarly $\binom{\tilde{a}}{k}$) can be computed as
 \begin{equation*}\label{binom2}
  \binom{a}{k} = \frac{a \  (a-1) \  (a-2)  \dots (a-k+1)}{k!}.
 \end{equation*}
 \noindent Next, since $s\leq 1$, 
 \begin{equation} \label{qs}
  \frac{q}{2}-\frac{2}{s} \leq \frac{q}{2} - 2. 
 \end{equation}
Therefore,
 \begin{align}\label{binom3}
   & \displaystyle \sum_{k=1}^{1 + k_{\frac{q}{2}-\frac{2}{s}} } \;
     \binom{\frac{q}{2}-\frac{2}{s} }{k-1} \; B(2k+1, \; q - 2k +1) \nonumber\\    
	&\qquad \qquad \leq \displaystyle \sum_{k=1}^{1+k_{\frac{q}{2} -2}} \; 
		\binom{\frac{q}{2} - 2}{k-1}   \; B(2k+1, \; q - 2k +1) \\ 
	&\qquad\qquad = \displaystyle \sum_{k=1}^{k_{\frac{q}{2}}} \; 
		\binom{\frac{q}{2} - 2}{k-1}   \; B\left(2k+1, \; 2\left(\frac{q}{2} - k\right) +1\right) .  \nonumber
\end{align}
Now applying \eqref{beta} yields \eqref{lem-sumB2-0}.
\end{prf}

\section{}

Finally, for completeness we include detailed calculation of deriving the representation of energies from \eqref{ep}. Recall that
\eqnn
v' = \frac{v+v_*}{2} + \frac{1}{2} |u| \sigma.
\eeqnn 
Hence,
\begin{align*}
	\langle v' \rangle^2 
	& =  1 + \frac{|v+v_*|^2}{4} + \frac{|v - v_*|^2}{4} + \frac{1}{2} |u| \sigma \\
	& = 1 + \frac{|v|^2 + |v_*|^2}{2} + \frac{1}{2} |u| (v + v_*) \cdot \left(\hat{u} \cos\theta + \omega \sin\theta\right) \\
	& =  1 + \frac{|v|^2 + |v_*|^2}{2} + \frac{1}{2} (v+v_*) \cdot (v- v_*) \cos\theta 
		+ \frac{1}{2} |u| |V| \sin\theta (\hat{V}\cdot \omega)\\
	& = 1 + |v|^2 \cos^2 \frac{\theta}{2} + |v_*|^2 \sin^2\frac{\theta}{2} 
		+ \frac{1}{2} |u||V| \sin\theta (j \cdot \omega) \sin\alpha \\
	& =  \langle v\rangle ^2 \cos^2 \frac{\theta}{2} + \langle v_*\rangle^2 \sin^2\frac{\theta}{2} 
		+ |v \times v_*| \sin\theta (j\cdot \omega),
\end{align*}
which coincides with the representation in \eqref{ep}.

\bigskip

{\bf{Acknowledgements.}} 
This work has been supported by  NSF under grants DMS-1413064, DMS-1516228, NSF-RNMS DMS1107465 and DMS 1440140.
 Support from both the Institute of Computational Engineering and Sciences (ICES) at the University of Texas Austin  and  the Mathematical Sciences Research Institute (MSRI) in Berkeley, California, are gratefully acknowledged.

\bibliographystyle{abbrv}
\bibliography{Majabib}

\begin{thebibliography}{10}

\bibitem{alcagamo13}
R.~Alonso, J.~A. Ca{\~n}izo, I.~Gamba, and C.~Mouhot.
\newblock A new approach to the creation and propagation of exponential moments
  in the {B}oltzmann equation.
\newblock {\em Comm. Partial Differential Equations}, 38(1):155--169, 2013.

\bibitem{alcaga10}
R.~J. Alonso, E.~Carneiro, and I.~M. Gamba.
\newblock Convolution inequalities for the {B}oltzmann collision operator.
\newblock {\em Comm. Math. Phys.}, 298(2):293--322, 2010.

\bibitem{alga08}
R.~J. Alonso and I.~M. Gamba.
\newblock Propagation of {$L^1$} and {$L^\infty$} {M}axwellian weighted bounds
  for derivatives of solutions to the homogeneous elastic {B}oltzmann equation.
\newblock {\em J. Math. Pures Appl. (9)}, 89(6):575--595, 2008.

\bibitem{alga11}
R.~J. Alonso and I.~M. Gamba.
\newblock Gain of integrability for the {B}oltzmann collisional operator.
\newblock {\em Kinet. Relat. Models}, 4(1):41--51, 2011.

\bibitem{allo13}
R.~J. Alonso and B.~Lods.
\newblock Two proofs of {H}aff's law for dissipative gases: the use of entropy
  and the weakly inelastic regime.
\newblock {\em J. Math. Anal. Appl.}, 397(1):260--275, 2013.

\bibitem{ar81}
L.~Arkeryd.
\newblock Intermolecular forces of infinite range and the {B}oltzmann equation.
\newblock {\em Arch. Rational Mech. Anal.}, 77(1):11--21, 1981.

\bibitem{bo84}
A.~V. Bobyl{\"e}v.
\newblock Exact solutions of the nonlinear {B}oltzmann equation and the theory
  of relaxation of a {M}axwell gas.
\newblock {\em Teoret. Mat. Fiz.}, 60(2):280--310, 1984.

\bibitem{bo97}
A.~V. Bobylev.
\newblock Moment inequalities for the {B}oltzmann equation and applications to
  spatially homogeneous problems.
\newblock {\em J. Statist. Phys.}, 88(5-6):1183--1214, 1997.

\bibitem{bogapa04}
A.~V. Bobylev, I.~M. Gamba, and V.~A. Panferov.
\newblock Moment inequalities and high-energy tails for {B}oltzmann equations
  with inelastic interactions.
\newblock {\em J. Statist. Phys.}, 116(5-6):1651--1682, 2004.

\bibitem{bo872}
L.~Boltzmann.
\newblock Weitere studien \"uber das w\"arme gleichgenicht unfer
  gasmol\"akuler.
\newblock {\em Sitzungsberichte der Akademie der Wissenschaften}, 66:275--370,
  1872. Translation : Further studies on the thermal equilibrium of gas
  molecules, in {\it Kinetic Theory} 2, 88-174, Ed. S.G. Brush, Pergamon,
  Oxford, 1966.

\bibitem{bo64}
L.~Boltzmann.
\newblock {\em Lectures on gas theory}.
\newblock Translated by Stephen G. Brush. University of California Press,
  Berkeley-Los Angeles, Calif., 1964.

\bibitem{ce88}
C.~Cercignani.
\newblock {\em The {B}oltzmann equation and its applications}, volume~67 of
  {\em Applied Mathematical Sciences}.
\newblock Springer-Verlag, New York, 1988.

\bibitem{de93}
L.~Desvillettes.
\newblock Some applications of the method of moments for the homogeneous
  {B}oltzmann and {K}ac equations.
\newblock {\em Arch. Rational Mech. Anal.}, 123(4):387--404, 1993.

\bibitem{de95}
L.~Desvillettes.
\newblock About the regularizing properties of the non-cut-off {K}ac equation.
\newblock {\em Comm. Math. Phys.}, 168(2):417--440, 1995.

\bibitem{dego00}
L.~Desvillettes and F.~Golse.
\newblock On a model {B}oltzmann equation without angular cutoff.
\newblock {\em Differential Integral Equations}, 13(4-6):567--594, 2000.

\bibitem{dewe04}
L.~Desvillettes and B.~Wennberg.
\newblock Smoothness of the solution of the spatially homogeneous {B}oltzmann
  equation without cutoff.
\newblock {\em Comm. Partial Differential Equations}, 29(1-2):133--155, 2004.

\bibitem{dibl74}
G.~Di~Blasio.
\newblock Differentiability of spatially homogeneous solutions of the
  {B}oltzmann equation in the non {M}axwellian case.
\newblock {\em Comm. Math. Phys.}, 38:331--340, 1974.

\bibitem{el83}
T.~Elmroth.
\newblock Global boundedness of moments of solutions of the {B}oltzmann
  equation for forces of infinite range.
\newblock {\em Arch. Rational Mech. Anal.}, 82(1):1--12, 1983.

\bibitem{emot53}
A.~Erd{\'e}lyi, W.~Magnus, F.~Oberhettinger, and F.~G. Tricomi.
\newblock {\em Higher transcendental functions. {V}ol. {III}}.
\newblock McGraw-Hill Book Company, Inc., New York-Toronto-London, 1955.
\newblock Based, in part, on notes left by Harry Bateman.

\bibitem{gapavi09}
I.~M. Gamba, V.~Panferov, and C.~Villani.
\newblock Upper {M}axwellian bounds for the spatially homogeneous {B}oltzmann
  equation.
\newblock {\em Arch. Ration. Mech. Anal.}, 194(1):253--282, 2009.

\bibitem{gapata15}
I.~M. Gamba, N.~Pavlovi\'{c}, and M.~Taskovi\'{c}.
\newblock Propagation of point-wise bounds for the boltzmann equation without
  cutoff.
\newblock {\em In preparation}, 2015.

\bibitem{gata09}
I.~M. Gamba and S.~H. Tharkabhushanam.
\newblock Spectral-{L}agrangian methods for collisional models of
  non-equilibrium statistical states.
\newblock {\em J. Comput. Phys.}, 228(6):2012--2036, 2009.

\bibitem{gr63}
H.~Grad.
\newblock Asymptotic theory of the {B}oltzmann equation. {II}.
\newblock In {\em Rarefied {G}as {D}ynamics ({P}roc. 3rd {I}nternat. {S}ympos.,
  {P}alais de l'{UNESCO}, {P}aris, 1962), {V}ol. {I}}, pages 26--59. Academic
  Press, New York, 1963.

\bibitem{li94}
P.-L. Lions.
\newblock On {B}oltzmann and {L}andau equations.
\newblock {\em Philos. Trans. Roy. Soc. London Ser. A}, 346(1679):191--204,
  1994.

\bibitem{lumo12}
X.~Lu and C.~Mouhot.
\newblock On measure solutions of the {B}oltzmann equation, part {I}: moment
  production and stability estimates.
\newblock {\em J. Differential Equations}, 252(4):3305--3363, 2012.

\bibitem{miwe99}
S.~Mischler and B.~Wennberg.
\newblock On the spatially homogeneous {B}oltzmann equation.
\newblock {\em Ann. Inst. H. Poincar\'e Anal. Non Lin\'eaire}, 16(4):467--501,
  1999.

\bibitem{mowaya16}
Y.~Morimoto, S.~Wang, and T.~Yang.
\newblock Measure valued solutions to the spatially homogeneous {B}oltzmann
  equation without angular cutoff.
\newblock {\em J. Stat. Phys.}, 165(5):866--906, 2016.

\bibitem{mo06}
C.~Mouhot.
\newblock Rate of convergence to equilibrium for the spatially homogeneous
  {B}oltzmann equation with hard potentials.
\newblock {\em Comm. Math. Phys.}, 261(3):629--672, 2006.

\bibitem{ol74}
F.~W.~J. Olver.
\newblock {\em Asymptotics and special functions}.
\newblock Academic Press [A subsidiary of Harcourt Brace Jovanovich,
  Publishers], New York-London, 1974.
\newblock Computer Science and Applied Mathematics.

\bibitem{po62}
A.~J. Povzner.
\newblock On the {B}oltzmann equation in the kinetic theory of gases.
\newblock {\em Mat. Sb. (N.S.)}, 58 (100):65--86, 1962.

\bibitem{vi98}
C.~Villani.
\newblock On a new class of weak solutions to the spatially homogeneous
  {B}oltzmann and {L}andau equations.
\newblock {\em Arch. Rational Mech. Anal.}, 143(3):273--307, 1998.

\bibitem{we96}
B.~Wennberg.
\newblock The {P}ovzner inequality and moments in the {B}oltzmann equation.
\newblock In {\em Proceedings of the {VIII} {I}nternational {C}onference on
  {W}aves and {S}tability in {C}ontinuous {M}edia, {P}art {II} ({P}alermo,
  1995)}, number 45, part II, pages 673--681, 1996.

\bibitem{we97}
B.~Wennberg.
\newblock Entropy dissipation and moment production for the {B}oltzmann
  equation.
\newblock {\em J. Statist. Phys.}, 86(5-6):1053--1066, 1997.

\end{thebibliography}

\end{document}